\documentclass[11pt]{article}
\usepackage{amscd,graphics,pgf,tikz,hyperref}
\usetikzlibrary{arrows}

\usepackage{amsfonts,amssymb,amscd,amsmath,enumerate,verbatim,calc,url}
\usepackage{amsthm, psfrag,latexsym,epsfig,mdwlist,graphicx}
\usepackage{times}  
\usepackage{hyperref}
\usepackage[capitalise, noabbrev]{cleveref}
\theoremstyle{plain}

\newtheorem{Theorem}{Theorem}[section]
\newtheorem{Lemma}[Theorem]{Lemma}
\newtheorem{Proposition}[Theorem]{Proposition}
\newtheorem{Corollary}[Theorem]{Corollary}
\theoremstyle{definition}
\newtheorem{Definition}[Theorem]{Definition}

\newtheorem{Remark}[Theorem]{Remark}
\newtheorem{Example}[Theorem]{Example}

\newtheorem{Question}[Theorem]{Question}

\newcommand{\coll}[1]{\searrow_{#1}}
\newcommand{\N}{\mathcal{N}}
\newcommand{\F}{\mathcal{F}}
\newcommand{\st}{\ : \ }
\newcommand{\sr}[1]{\wr_{#1}} 
\newcommand{\bfx}{\mathbf{x}}
\newcommand{\wh}{\widetilde{H}}

\newcommand{\link}{\mathrm{link}}
\newcommand{\sm}{\setminus}

\makeatletter
\newcommand{\subjclass}[2][1991]{%
  \let\@oldtitle\@title%
  \gdef\@title{\@oldtitle\footnotetext{#1 \emph{Mathematics subject classification.} #2}}%
}
\newcommand{\keywords}[1]{%
  \let\@@oldtitle\@title%
  \gdef\@title{\@@oldtitle\footnotetext{\emph{Key words and phrases.} #1.}}%
}
\makeatother

\def\C{\mathcal{C}}
\def\Simp{\mathrm{Simp}}
\def\facets{\mathrm{Facets}}

\def\xb{{\bold x}}

\begin{document}

 \title{Chordality, $d$-collapsibility, and
  componentwise linear ideals}

 \author{Mina Bigdeli{\thanks{School of Mathematics, Institute for
     Research in Fundamental Sciences (IPM), P. O. Box: 19395-5746, Tehran, Iran,
     mina.bigdeli98@gmail.com, mina.bigdeli@ipm.ir. Research supported by IPM  and also partially supported by National Science Foundation under Grant No. DMS-1440140 while the first author was resident at MSRI during the Spring 2017 semester.}}{}~ and  Sara Faridi\thanks{Department of
     Mathematics and Statistics, Dalhousie University, Halifax,
     Canada, faridi@mathstat.dal.ca. Research supported by NSERC.}}
     
\subjclass[2010]{Primary 13D02, 13F55; Secondary 05E45, 05E40.}
\keywords{Betti number, Chordality, Collapsibility, Componentwise linear,  Linear resolution, Simplicial Complex, Stanley-Reisner ideal}

\maketitle


\begin{abstract} 
    Using the concept of $d$-collapsibility from combinatorial
    topology, we define chordal simplicial complexes and show that
    their Stanley-Reisner ideals are componentwise linear. Our
    construction is inspired by and an extension of ``chordal
    clutters'' which was defined by Bigdeli, Yazdan Pour and
    Zaare-Nahandi in 2017,  and characterizes Betti tables
      of all ideals with linear resolution in a polynomial ring.

   We show $d$-collapsible and $d$-representable complexes produce
componentwise linear ideals for appropriate $d$.  Along the way,
we prove that there are generators that when added to the ideal, do not
change Betti numbers in certain degrees.

 We then show that large classes of componentwise linear ideals, such
 as Gotzmann ideals and square-free stable ideals have chordal
 Stanley-Reisner complexes, that Alexander duals of vertex
 decomposable complexes are chordal, and conclude that the Betti table
 of every componentwise linear ideal is identical to that of the
 Stanley-Reisner ideal of a chordal complex.
    \end{abstract}


\section*{Introduction}

Chordal simplicial complexes, as we call them here, arise from work of
Bigdeli, Yazdanpour and Zaare-Nahandi~\cite{BYZ} in 2017, where they
defined chordal clutters in an attempt to give a combinatorial
description of square-free monomial ideals that have linear resolution
over all fields. The term ``chordal" and the general approach stem
from Fr\"oberg's 1990 paper~\cite{Fr} in which ideals generated by
degree 2 monomials are characterized in terms of chordal
graphs. Fr\"oberg's work initiated investigations by many authors find
similar criteria for ideals with linear resolution generated by
monomials of higher degree, which led to generalizations of
chordality: the classes defined by Van Tuyl and Villarreal~\cite{VTV}
in~2008, Emtander~\cite{Em} in~2010, Woodroofe~\cite{Wo} in~2011, all
produce ideals with linear resolution over all fields, and all these
classes were shown to be contained in the class of chordal simplicial
complexes in~\cite{BYZ} (which we later found is
  equivalent to a class of simplicial complexes appearing in Cordovil,
  Lemos, and Sales~\cite{CLS} in~2009).

 On the other hand Connon and Faridi~\cite{CF1} in 2013
  gave a more general definition of chordality by focusing on
  necessary conditions for vanishing of simplicial homology, which
  forced a simplicial complex producing a linear resolution in any
  characteristic to belong to their class, and a more restrictive
  definition in~\cite{CF2} in 2015 characterized all simplicial
  complexes whose ideals have linear resolution over fields of
  characteristic~2. Adiprasito, Nevo, and Samper's
  work~\cite{ANS} in 2016 characterized chordality by checking a
  smaller interval for the vanishing of simplicial homology, giving
  a homological characterization of chordality.

  Since betti numbers depend on the characteristic of the ground
  field, for a combinatorial characterization of chordality, one
  should expect a definition that produces ideals that have linear
  resolution over all fields.  So far neither of the above classes
  combinatorially characterizes monomial ideals with linear resolution,
  even when one considers ideals that have linear resolution over all
  fields.

  However, it was shown by Bigdeli, Herzog, Yazdanpour and
  Zaare-Nahandi~\cite{BHYZ} that every Betti table of a graded
  ideal with linear resolution is the Betti table of an ideal coming
  from a chordal clutter, as defined in~\cite{BYZ}.

In this paper, we adapt the concept of chordal clutters
  from~\cite{BYZ} and change the perspective from clutters to simplicial
complexes. As a result, we show that chordality of the
Stanley-Reisner complex of an ideal generated in degree $d+1$ is
equivalent to $d$-collapsibility, a notion well-known and well-used in
algebraic topology and combinatorics which has specific homological
consequences. Among other things, this perspective allows us to:

\begin{itemize}

\item show that $d$-chordal simplicial complexes (one of the largest known classes of
  complexes which produce ideals with linear resolution over all fields) are
  essentially, but not exactly,  the same as $d$-collapsible ones
  (\cref{d-collapsible=chordal});

\item introduce a large class of complexes, which we call
  \emph{chordal complexes}, whose Stanley-Reisner ideals are
  componentwise linear (\cref{main1});

\item show that, for a suitable $d$, $d$-collapsible and
  $d$-representable simplicial complexes  are chordal
  and have componentwise linear Stanley-Reisner ideals (\cref{main1});

\item show that square-free stable monomial ideals have chordal
  Stanley-Reisner complexes (\cref{SS});

\item show that Alexander duals of vertex decomposable complexes are
  chordal (\cref{VD->Chordal});

\item show that Gotzmann square-free monomial ideals have chordal
  Stanley-Reisner complexes (\cref{Gotz});

\item show that Betti tables of Stanley-Reisner ideals of chordal
  complexes encompass all Betti tables of componentwise linear ideals
  (\cref{BT-CWL});

\item show that  there are specific monomials we can
  add to the generators of a monomial ideal without affecting the Betti
  numbers  in most degrees (\cref{betti numbers of chordals});

\item using induced subcomplexes, find useful inductive properties of
  componentwise linear ideals (\cref{main1}).

\end{itemize}

The authors are grateful for helpful comments from Eran Nevo and
Mayada Shahada , and for the hospitality of The Simons Institute for
the Theory of Computing in California, where they started this work in 2016.


\section{Basic definitions}

A {\bf simplicial complex} $\Gamma$ on the vertex set $[n]=\{
1,\ldots,n \}$, is a set of subsets of $[n]$ such that if $F\in
\Gamma$ and $F'\subseteq F$, then $F'\in \Gamma$.  Each element of
$\Gamma$ is called a {\bf face} of $\Gamma$. A {\bf facet} is a
maximal face of $\Gamma$ (with respect to inclusion). The dimension of
a face $F$ is $\dim F=|F|-1$. We define $\dim \emptyset=-1$. A face
$F$ of $\Gamma$ with $\dim F=t$ is called a $t$-face of $\Gamma$. Let
$d = \max\{\dim F : F \in \Gamma\}$ and define the dimension of
$\Gamma$ to be $\dim \Gamma=d$. We say that a simplicial complex is
{\bf pure} if all its facets have the same dimension.

A simplicial complex $\Gamma$ is uniquely determined by its facets. We
denote the set of the facets of $\Gamma$ by $\facets(\Gamma)$ and when
$\facets(\Gamma)=\{F_1, \ldots, F_m\}$, we write $\Gamma=\langle
F_{1}, \ldots, F_{m} \rangle$.  A simplicial complex with only one
facet is called a {\bf simplex}.

A {\bf subcomplex} $\Sigma$ of $\Gamma$ is a simplicial complex with $\Sigma\subset \Gamma$. Let $E\subset [n]$. By $\Gamma\sm E$ we mean
$$\Gamma\sm E= \{F\in \Gamma \st E\not\subseteq F \}$$ 
which is a subcomplex of $\Gamma$. 
If $W\subset [n]$, we denote by $\Gamma_W$ the {\bf induced
  subcomplex} of $\Gamma$ on the set $W$, in other
words $$\Gamma_W=\{F\in \Gamma \st F\subset W\}.$$

 The {\bf Alexander dual} $\Gamma^\vee$ of $\Gamma$ is
 the simplicial complex $$\Gamma^\vee=\{ F \subseteq [n] \st
 [n]- F \notin \Gamma \}.$$

 If $F$ is a face of $\Gamma$, then $\link_\Gamma(F)$ is the
 simplicial complex on $[n]- F$ defined
 as $$\link_\Gamma(F)= \{G \in \Gamma \st F \cap G =\emptyset 
  \mbox{ and } G\cup F \in \Gamma\}.$$

For a  nonnegative integer $i \leq \dim \Gamma$, we define the {\bf pure
  $i$-skeleton} $\Gamma^{[i]}$ of $\Gamma$ to be the simplicial complex
 $$\Gamma^{[i]}=\langle F\in \Gamma \st \dim F =i \rangle.$$

 A {\bf nonface} of $\Gamma$ is a subset $F$ of $[n]$ with $F\notin
 \Gamma$.

\begin{Definition}[\bf Stanley-Reisner ideal/complex]\label{d:SR} 
  Let $S=K[x_1,\ldots,x_n]$ be the polynomial ring over the field $K$
  with $n$ indeterminates.

\begin{itemize}

\item Let $\Gamma$ be a simplicial complex on $n$ vertices.  The 
{\bf Stanley-Reisner ideal} of $\Gamma$ is the monomial ideal
  $\N(\Gamma)$ of $S$ which is generated by the square-free monomials
  $\xb_F:=\prod_{i\in F} x_i$ with $F\notin \Gamma$. In other
  words $$\N(\Gamma)=(\xb_F \st F\notin \Gamma).$$ The {\bf
    Stanley-Reisner ring}, $K[\Gamma]$, is defined to be the quotient
  ring $S/\N(\Gamma)$.

\item Let $I$ be a square-free monomial ideal in $S$. We define its {\bf
    Stanley-Reisner complex} $\N(I)$ to be the simplicial
  complex $$\N(I)=\{F\subseteq [n] \st \xb_F \notin I\}.$$
\end{itemize}
It follows directly from the definitions that the Stanley-Reisner
correspondence is a one-to-one correspondence between simplicial
complexes on the vertex set $[n]$ and square-free monomial ideals in $S$.
\end{Definition}

Let $I\in S=K[x_1,\ldots, x_n]$ be a graded ideal and let
$$\F: 0\longrightarrow F_p\longrightarrow\ldots\longrightarrow F_1\longrightarrow F_0\longrightarrow I\longrightarrow 0$$
 be its graded minimal free resolution with  $F_i =\bigoplus_{j} {S}^{\beta^s_{i,j}(I)}(-j)$, for all $i$. For any pair of integers $(i,j)$, the {\bf $(i,j)$-th graded Betti number} of $I$ in $S$ is defined to be
	$$\beta^S_{i,j}(I) = \dim_K \mathrm{Tor}_{i}^{S}(K,I)_j$$
	for all $i$ and $j$. Throughout, we write $\beta_{i,j}(I)$  for $\beta^S_{i,j}(I)$. 	The ideal $I$ is called to have {\bf $d$-linear resolution} if $\beta_{i,j}(I)=0$ for all $i$ and all $j$ with $j\neq i+d$.


\section{d-chordality}
The definition below is a slight variation of that given
in~\cite[Definition~5.4]{CF1}.
\begin{Definition}[\bf $d$-closure]\label{d:d-closure}
  Let $\Gamma$ be a simplicial complex on the vertex set $[n]$ and $d$
  a positive integer. The $d$-closure of $\Gamma$, denoted by
  $\Delta_d(\Gamma)$, is the simplicial complex on $[n]$ whose faces
  are given in the following way:
	\begin{enumerate}
        \item[$\bullet$] the $d$-faces of $\Delta_d(\Gamma)$ are
          exactly the $d$-faces of $\Gamma$;
        \item[$\bullet$] all subsets of $[n]$ with at most $d$
          elements are faces of $\Delta_d(\Gamma)$;
	\item[$\bullet$] a subset of $[n]$ with more than $d + 1$
          elements is a face of $\Delta_d(\Gamma)$ if and only if all
          of its subsets of $d + 1$ elements are faces of $\Gamma$.
			\end{enumerate}

If $\Gamma$ is the $d$-closure of a simplicial complex, 
we simply say that  $\Gamma$ is a {\bf $d$-closure}.
\end{Definition}

To justify this terminology, note that all the simplicial complexes on
$[n]$ which have the same pure $d$-skeleton, have the same
$d$-closure. In particular, if $\Gamma= \Delta_d(\Sigma)$, by definition we have 
$\Gamma^{[d]}=\Sigma^{[d]}$ and  it follows that 
\begin{center} $\Gamma=\Delta_d(\Sigma)$
$\iff$ $\Delta_d(\Sigma)=\Delta_d(\Gamma)$
$\iff$ $\Gamma=\Delta_d(\Gamma).$
\end{center}

\begin{Example}\label{Example0}
  Let $\Gamma=\langle \{2,5\},\{1,4,5\},\{1,2,3,4\} \rangle$ be a
  simplicial complex on $[5]$ in \cref{Gamma chordal}.  
  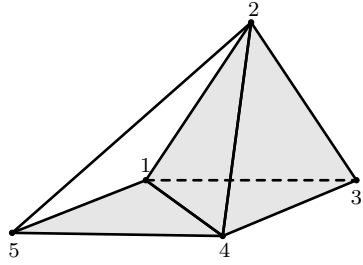
\begin{figure}[ht!]
\begin{center}
\begin{tikzpicture}[line cap=round,line join=round,>=triangle 45,x=0.7cm,y=0.7cm]
\clip(9.,4.5) rectangle (16.5,9.4);
\fill[line width=1.2pt,fill=black,fill opacity=0.10000000149011612] (14.,9.) -- (12.,6.) -- (13.46,4.94) -- cycle;
\fill[line width=2.pt,fill=black,fill opacity=0.10000000149011612] (14.,9.) -- (16.,6.) -- (13.46,4.94) -- cycle;
\fill[line width=2.pt,fill=black,fill opacity=0.10000000149011612] (12.,6.) -- (9.46,5.) -- (13.46,4.94) -- cycle;
\draw [line width=1.pt] (14.,9.)-- (12.,6.);
\draw [line width=1.pt] (12.,6.)-- (13.46,4.94);
\draw [line width=1.pt] (13.46,4.94)-- (14.,9.);
\draw [line width=1.pt] (14.,9.)-- (16.,6.);
\draw [line width=1.pt] (16.,6.)-- (13.46,4.94);
\draw [line width=1.pt] (13.46,4.94)-- (14.,9.);
\draw [line width=1.pt,dash pattern=on 3pt off 3pt] (12.,6.)-- (16.,6.);
\draw [line width=1.pt] (12.,6.)-- (9.46,5.);
\draw [line width=1.pt] (9.46,5.)-- (13.46,4.94);
\draw [line width=1.pt] (13.46,4.94)-- (12.,6.);
\draw [line width=1.pt] (14.,9.)-- (9.46,5.);
\draw (9.2,5) node[anchor=north west] {\begin{scriptsize}$5$\end{scriptsize}};
\draw (11.7,6.6) node[anchor=north west] {\begin{scriptsize}$1$\end{scriptsize}};
\draw (13.2,5) node[anchor=north west] {\begin{scriptsize}$4$\end{scriptsize}};
\draw (13.75,9.57) node[anchor=north west] {\begin{scriptsize}$2$\end{scriptsize}};
\draw (15.7,6) node[anchor=north west] {\begin{scriptsize}$3$\end{scriptsize}};
\begin{scriptsize}
\draw [fill=black] (14.,9.) circle (1.0pt);
\draw [fill=black] (12.,6.) circle (1.0pt);
\draw [fill=black] (13.46,4.94) circle (1.0pt);
\draw [fill=black] (16.,6.) circle (1.0pt);
\draw [fill=black] (9.46,5.) circle (1.0pt);
\end{scriptsize}
\end{tikzpicture}
\caption{The simplicial complex $\Gamma$}\label{Gamma chordal}
\end{center}
\end{figure}

Note
  that $\dim(\Gamma)=3$.
 We have 
\begin{align*}
\Delta_1(\Gamma)&=\langle \{1,2,4,5\},\{1,2,3,4\}\rangle,\\
\Delta_2(\Gamma)&=\langle \{2,5\}, \{3,5\},\{1,4,5\},\{1,2,3,4\}\rangle,\\
\Delta_3(\Gamma)&=\langle \{1,2,5\}, \{1,3,5\},\{1,4,5\},\{2,3,5\},\{2,4,5\},\{3,4,5\},\{1,2,3,4\}\rangle,\\
\Delta_i(\Gamma)&=\langle [5]\rangle^{[i-1]}, \quad\text{ for }i\geq 4.
\end{align*}
\end{Example}

It is shown in \cite[Proposition~5.6]{CF1} that a square-free monomial
ideal $I$ is equigenerated in degree $d+1$ if and only if $\N(I)$ is a
$d$-closure, i.e.
$$\N(I)=\Delta_{d}(\N(I)).$$

\begin{Definition}[{\bf free face and simplicial face} {\cite[Definition~2.13]{MNYZ}}]\label{simp}
  A  face $E$ of a simplicial complex $\Gamma$
  is called a {\bf free face} if it appears in a unique facet of
  $\Gamma$.   Note that facets are automatically free faces. 
  
  If $\Gamma$ is a $d$-closure and
  $\dim E=d-1$, then this free face is called {\bf simplicial}.  We
  denote the set of all simplicial faces of $\Gamma$ by
  $\mathrm{Simp}(\Gamma)$.
\end{Definition}

Let $\Gamma$ be a simplicial complex on $[n]$ and $E \subset [n]$. The
{\bf deletion} of $E$ from $\Gamma$, is the simplicial complex 
$$
\Gamma\sr{E}=\{F\in \Gamma: \ E\subsetneq F\}=\begin{cases}\Gamma &\text{ if } E\notin \Gamma\\
(\Gamma \sm E) \cup \{E\} &\text{ if } E\in \Gamma.
\end{cases}
$$
Note that if $E\in \Gamma$, the face $E$ is not deleted in this
operation.  In case $E$ is a simplicial face of a $d$-closure
$\Gamma$, this operation is called {\bf simplicial deletion} of $E$
from $\Gamma$. The simplicial complex obtained from a simplicial
deletion is again a $d$-closure.  Note also that all $(d-1)$-faces of
a $d$-closure $\Gamma$ which are its facets are simplicial. Indeed, for a $d$-closure $\Gamma$,
$\Gamma\sr{E}=\Gamma$ if and only if $E$ is a facet of $\Gamma$.

Let $\mathbf{E}=E_1,\ldots, E_t$ be a sequence of $(d-1)$-faces of a
$d$-closure $\Gamma$. The sequence $\mathbf{E}$ is called a {\bf
  simplicial sequence} of $\Gamma$ if $E_1\in \mathrm{Simp}(\Gamma)$,
and $E_i\in \mathrm{Simp}(\Gamma\sr{E_1}\ldots\sr{E_{i-1}})$ for all $i\geq 2$.  The sequence $\mathbf{E}$ is called a
{\bf simplicial order} of $\Gamma$ if $E_1$ is not a facet of
$\Gamma$, $E_i$ is not a facet in $\Gamma\sr{E_1}\ldots\sr{E_{i-1}}$ and
$$\Gamma\sr{E_1}\ldots\sr{E_{t}}=\langle [n]\rangle ^{[d-1]}.$$
In order to shorten the notation, we often use $\Gamma\sr{E_1,\ldots,E_{t}}$ instead of $\Gamma\sr{E_1}\ldots\sr{E_{t}}$.

\begin{Example}\label{Example1}
  Consider $\Delta_2(\Gamma)$ in~\cref{Example0} and let
  $E_1=\{1,5\}$. Since $E_1$ is uniquely contained in the facet
  $\{1,4,5\}$, it is a simplicial face of $\Delta_2(\Gamma)$. We have
\[
\Sigma_1:=\Delta_2(\Gamma)\sr{E_1}=\langle \{1,5\},\{2,5\},\{3,5\},\{4,5\},\{1,2,3,4\} \rangle.
\]
Now let 
$E_2=\{1,2\}$. Since the only facet in $\Sigma_1$ containing $E_2$ is $\{1,2,3,4\}$, $E_2$ is simplicial in $\Sigma_1$. Then
\[
\Sigma_2:=\Sigma_1\sr{E_2}=\langle \{1,2\},\{1,5\},\{2,5\},\{3,5\},\{4,5\},\{1,3,4\},\{2,3,4\} \rangle.
\]
Now $E_3=\{1,3\}$ is simplicial in $\Sigma_2$ and 
\[
\Sigma_3:=\Sigma_2\sr{E_3}=\langle \{1,2\},\{1,3\},\{1,4\},\{1,5\},\{2,5\},\{3,5\},\{4,5\},\{2,3,4\} \rangle.
\]
Finally $E_4=\{2,3\}$ is simplicial in $\Sigma_3$ and
\begin{align*}
\Sigma_3\sr{E_4}&=\langle \{1,2\},\{1,3\},\{1,4\},\{1,5\},\{2,3\},\{2,4\},\{2,5\},\{3,4\},\{3,5\},\{4,5\}\rangle\\
&=\langle [5]\rangle^{[1]}.
\end{align*}
Therefore $E_1,\ldots,E_4$ is a simplicial order of $\Delta_2(\Gamma)$. 
\end{Example}

\begin{Lemma}\label{l:easy} Let $d$ be a positive integer, $\Gamma$ 
  a simplicial complex and $E$ a $(d-1)$-dimensional face of
  $\Delta_d(\Gamma)$. Then $\Delta_d(\Gamma)\sr{E}=\Delta_d(\Gamma\sr{E})$.
\end{Lemma}

\begin{proof} It is clear that the two complexes have the same faces
  of dimension $\leq d-1$. By~\cref{d:d-closure}, if $\dim(F)=d$
  $$F \in \Delta_d(\Gamma)\sr{E} 
  \iff F \in \Gamma \mbox{ and }F \not\supset E 
  \iff F \in\Gamma\sr{E} 
  \iff F \in \Delta_d(\Gamma\sr{E})$$
   and if   $\dim(F)>d$
 \begin{align*}
 F \in \Delta_d(\Gamma)\sr{E} 
  &\iff \forall \ G \subset F \mbox{ if } \dim(G)=d \mbox{ then } G \in \Gamma \mbox{ and }F \not\supset E\\ 
 & \iff \forall \ G \subset F \mbox{ if } \dim(G)=d \mbox{ then } G \in \Gamma \mbox{ and }G \not\supset E\\ 
  &\iff \forall \ G \subset F \mbox{ if } \dim(G)=d \mbox{ then } G  \in \Gamma\sr{E} \\
  &\iff F \in \Delta_d(\Gamma\sr{E}).
\end{align*}
    \end{proof}

    \Cref{l:easy} allows us to define a chordal simplicial
    complex with two equivalent conditions. Below we define chordal
    simplicial complexes using the concept of chordal clutters as
    defined by the first author and the coauthors in~\cite{BYZ}.

\begin{Definition}[{\bf$d$-chordal and chordal simplicial complex}, see\cite{BYZ}]\label{d:chordal}
  Let $\Gamma$ be a simplicial complex on the vertex set $[n]$ and $d$
  a positive integer. We say that $\Gamma$ is {\bf $d$-chordal} if it
  satisfies one of the following equivalent conditions:

	\begin{itemize}
        \item[$(*)$] either $\Delta_{d}(\Gamma)=\langle
          [n]\rangle ^{[d-1]}$, or else $\Delta_d(\Gamma)$
          admits a simplicial order.

        \item[$(*')$] either $\Delta_{d}(\Gamma)=\langle
          [n]\rangle ^{[d-1]}$, or else there is $E\in
          \mathrm{Simp}(\Delta_{d}(\Gamma))$ such that $E$ is
          not a facet of $\Delta_d(\Gamma)$ and $\Gamma\sr{E}$ satisfies
          condition $(*')$.
       \end{itemize}

       We say that $\Gamma$ is {\bf chordal} if it is $d$-chordal for
       every $d \geq 1$.
\end{Definition}

Later in \cref{the bounds for checking chordality} we will show that
to prove $\Gamma$ is chordal, it is sufficient to check it is
$d$-chordal for a finite number of values of $d$.

 \Cref{d:chordal} of a ``$d$-chordal simplicial complex''
  is a Stanley-Reisner equivalent of ``chordal ($d+1$)-uniform clutters''
  in~ \cite{BYZ}.  The following statement follows directly from the
  definitions,  we include it for the sake of comparison.

\begin{Proposition} 
Let $\Gamma$ be a simplicial complex on $[n]$, and let $\mathcal{C}$
be the $(d+1)$-uniform clutter $$\mathcal{C}=\{ A \subseteq [n] \st |A|=d+1,
A \in \Gamma \}=\facets(\Gamma^{[d]}).$$ Then $\mathcal{C}$ is chordal in the sense of
\cite{BYZ} if and only if $\Gamma$ is $d$-chordal.
\end{Proposition}

In particular, as in the case of~\cite{BYZ}, our definition of chordality
for simplicial complexes extends that of graphs. Given a simplicial
complex $\Gamma$, its $1$-closure $\Delta_{1}(\Gamma)$ is the clique
complex of a graph $G=\Gamma^{[1]}$.  It is clear that $G$ is chordal
(i.e. has no minimal cycles of length greater than $3$) if and only if
its clique complex $\Delta_{1}(\Gamma)=\Delta_1(G)$ is $1$-chordal.


\section{$d$-collapsing}

In this section we show how the concept of {\bf elementary
  $d$-collapsing} introduced by Wegner~\cite{We} relates directly to
simplicial deletion. Elementary $d$-collapsing is a
  special case of the better known operation of \emph{simplicial
    collapsing} (see for example~\cite[Definition~6.13]{K}), which
  when applied to a simplicial complex produces a new simplicial
  complex which is homotopy equivalent to the original one. The main
  difference between the two operations is that in the case of
  $d$-collapsing a free face is allowed to be facet.

Recall that $\Gamma\sm{E}$  refers to the operation of deleting
all faces of the simplicial complex $\Gamma$ containing the face $E$
(including $E$ itself). In the case where $E$ is a free face we 
denote this complex by $\Gamma\coll{E}$, that is $$ \Gamma\coll{E}=\Gamma\sm{E}.$$

  A sequence of faces $\mathbf{E}=E_1,\ldots, E_t$ is called a {\bf
    free sequence} of $\Gamma$ if $E_1$ is a free face in $\Gamma$,
  and $E_i$ is a free face in $\Gamma\coll{E_1}\ldots\coll{E_{i-1}}$
  for all $i>1$.  We shorten the notation for the series of deletions,
  by using
  $$\Gamma\coll{E_1,\ldots,E_t}=\Gamma\coll{E_1}\cdots\coll{E_{t}}.$$

\begin{Definition}[\bf $d$-collapsing]\label{d-collapsing}
  If $\Gamma$ is a simplicial complex with a free face $E$, and $d$ is a
  positive integer with $\dim E<d$, then the operation
  $\Gamma\coll{E}$ is called an {\bf elementary $d$-collapsing.}  The
  simplicial complex $\Gamma$ is called {\bf $d$-collapsible} if it
  can be reduced to the void complex $\emptyset$
  after a finite number of elementary $d$-collapsings.
\end{Definition}

Suppose now $\Gamma$ is a $d$-closure and $E$ is a simplicial face of
$\Gamma$. Then, by definition, $E$ is a free face with $\dim E=d-1$ and 
\begin{align*}
\Gamma\sr{E}=\Gamma\coll{E}\cup \{E\}.
\end{align*}  
Suppose ${\bf
  E}=E_1,\ldots,E_t$ is a simplicial order of $\Gamma$. Then it is a free sequence of $\Gamma$ and 
\begin{align}\label{e:simp-order-collapse}
\Gamma\coll{E_1,\ldots,E_t}=\langle [n]\rangle
^{[d-1]}- \{E_1,\ldots, E_t\}.
\end{align}

We now start working our way towards \cref{d-collapsible=chordal},
where we show that there is a direct relation between the $d$-chordal
simplicial complexes and $d$-collapsible ones. 

A very useful tool when considering $d$-collapsings is
\cref{l:experimental} below, which we proved independently and then
found later in Tancer's work~\cite{T}. We refer the reader there for a
full proof.

\begin{Lemma}[Tancer~\cite{T}, Lemma~5.1]\label{l:experimental} Let 
  $\Sigma$ be a simplicial complex, $d$ a positive integer, $E
  \subsetneq E' $ free faces of $\Sigma$ of dimension $<d$. Then
  $\Sigma \coll{E'}$ $d$-collapses to $\Sigma \coll{E}$. In
  particular, if $\Sigma \coll{E}$ is $d$-collapsible, then so is
  $\Sigma \coll{E'}$.
\end{Lemma}

\begin{Lemma}\label{changing the position of a simp}
Let $\Sigma$ be a simplicial complex on the vertex set $[n]$ and let $\mathbf{E}=E_1,\ldots, E_r$ be a free sequence of $\Sigma$  with the property that $\dim E_r=d-1$ and $E_r$ is the only element in this sequence such that the unique facet containing it has dimension $\geq d$. Then  $E_r$ is a free face of $\Sigma$ and $E_1,\ldots,E_{r-1}$ is a free sequence of $\Sigma\coll{E_r}$. Moreover, $$\Sigma\coll{E_r, E_1,\ldots,E_{r-1}}=\Sigma\coll{E_1,\ldots,E_{r}}.$$
\end{Lemma}

\begin{proof}
  Suppose $F$ is the unique facet in
  $\Sigma\coll{E_1,\ldots,E_{r-1}}$ which contains
  $E_r$. 
   If $T\in\Sigma$ has dimension$\geq d$, then $E_i\not\subset T$ for
  $i<r$. Thus $T\in\Sigma\coll{E_1,\ldots,E_{r-1}}$. It follows that
  $F$ is a facet in $\Sigma$. Suppose $G$ is another facet in $\Sigma$
  containing $E_r$. Since $F$ is the unique facet in
  $\Sigma\coll{E_1,\ldots,E_{r-1}}$ containing $E_r$, we conclude that
  $G$ contains some $E_i$ with $i<r$. Hence $\dim G<d$ and so
  $G=E_r\subset F$, a contradiction. Thus $E_r$ is a free face of
  $\Sigma$ .

  Now we show that $E_1,\ldots,E_{r-1}$ is a free sequence of
  $\Sigma\coll{E_r}$. Suppose $F_1$ is the unique facet in $\Sigma$
  containing $E_1$ and for $1<i<r$, $F_i$ is the unique facet in
  $\Sigma\coll{E_1,\ldots,E_{i-1}}$ containing $E_i$. Then since $\dim
  F_i<d$, we have $E_r\subseteq F_i$ if and only if $E_r=F_i$. If
  $E_r\subset F_i$, then $F_i\subset F$ and since $F\in
  \Sigma\coll{E_1,\ldots,E_{i-1}}$, it contradicts the fact that $F_i$
  is a facet. Therefore $E_r\not\subseteq F_i$ for $1\leq i\leq r-1$. Hence $F_1$ is a facet in $\Sigma\coll{E_r}$ and
  $F_i$ is a facet in $\Sigma\coll{E_r}\coll{E_1,\ldots,E_{i-1}}$ for
  $1<i<r$. Moreover, since
  $\Sigma\coll{E_r}\coll{E_1,\ldots,E_{i-1}}\subseteq
  \Sigma\coll{E_1,\ldots,E_{i-1}}$ it follows that $F_i$ is the only
  facet in $\Sigma\coll{E_r, E_1,\ldots,E_{i-1}}$ containing
  $E_i$. Thus $E_1$ is a free face in $\Sigma\coll{E_r}$ and $E_i$ is
  a free face in $\Sigma\coll{E_r, E_1,\ldots,E_{i-1}}$.
\end{proof}

An immediate consequence of \cref{l:experimental} and
\cref{changing the position of a simp} is that $d$-collapsibility and $d$-chordality are
intimately connected.

\begin{Theorem}[\bf $d$-collapsible is equivalent to $d$-chordal for
  $d$-closures]\label{d-collapsible=chordal}
  Let $\Gamma$ be a $d$-closure on the vertex set $[n]$ for a positive
  integer $d$.  Then $\Gamma$ is $d$-chordal if and only if $\Gamma$
  is $d$-collapsible. In particular, a simplicial complex $\Sigma$ is
  chordal if and only if for all $d\geq 1$ the simplicial complex
  $\Delta_d(\Sigma)$ is $d$-collapsible.
\end{Theorem}

  \begin{proof} It is enough to prove the first statement. Suppose
  $\Gamma$ is $d$-collapsible. We prove by induction on the number of
  faces of $\Gamma$ that $\Gamma$ is $d$-chordal.  The base case of
  the induction is the smallest $d$-closure $\Gamma=\langle [n]
  \rangle^{[d-1]}$ which is $d$-chordal by definition.

  Suppose $\Gamma \supsetneq\langle [n] \rangle^{[d-1]}$ is
  $d$-collapsible. Hence there is a free sequence
  $E_1,\ldots,E_t$ such that $\dim E_i<d$ for all $i$ and
  $\Gamma\coll{E_1,\ldots,E_t}=\emptyset$. Suppose $r$
  is the smallest integer in $1,\ldots,t$ such that the  facet
  $F$ of $\Gamma\coll{E_1,\ldots,E_{r-1}}$ uniquely containing $E_r$ has
  dimension $\geq d$. Note that since $\Gamma \neq\langle [n] \rangle^{[d-1]}$, such $r$ exists. We may assume by \cref{l:experimental} that
  $\dim E_r=d-1$. By \cref{changing the position of a simp} we know that $E_r\in
  \Simp(\Gamma)$ and $E_1,\ldots,E_{r-1}$ is a free sequence of
  $\Gamma\coll{E_r}$. 

  Since
  $\Gamma\coll{E_r}\coll{E_1,\ldots,E_{r-1}}=\Gamma\coll{E_1,\ldots,E_{r}}$
  is $d$-collapsible, so is $\Gamma\coll{E_r}$. On the other hand
  $\Gamma\coll{E_r}=(\Gamma\sr{E_r})\coll{E_r}.$  
  So $\Gamma\sr{E_r}$ is $d$-collapsible, and hence $d$-chordal by the
  induction hypothesis, which implies that $\Gamma$ is
  $d$-chordal.

 Suppose now that $\Gamma$ is $d$-chordal and admits a simplicial order
  $E_1,\ldots,E_t$.  By \cref{e:simp-order-collapse} the
  simplicial complex $\Gamma$, $d$-collapses along this sequence to
  $\langle [n]\rangle ^{[d-1]}-\{E_1,\ldots, E_t\}$, and since
  all faces of this simplicial complex have dimension $<d$, it
  collapses into $\langle [n]\rangle^{[d-2]}$ by elementary
  $d$-collapsings along its facets. Continuing this process one sees
  that $\langle [n]\rangle ^{[d-1]}-\{E_1,\ldots, E_t\}$
  collapses into $\emptyset$ by a sequence of
  elementary $d$-collapsings. Hence $\Gamma$ is $d$-collapsible.
    \end{proof}

	 Given a simplicial complex $\Gamma$ on $[n]$, the set of all
         $d$-faces of $\Gamma$ forms a ``$(d+1)$-uniform clutter'' which we call
         $\C$. In~\cite{BYZ} the authors defined a chordal $(d+1)$-uniform clutter.
         It is straightforward to check that $\Gamma$ is $d$-chordal if
         and only if $\C$ is   a chordal $(d+1)$-uniform clutter.  It is
         also proved in~\cite{BYZ} that if $\Gamma=\langle
         [n]\rangle$, then $\C$ is chordal. It follows that $\Gamma=\langle
         [n]\rangle$  is $d$-chordal for any $d$. In the following lemma we give a
         direct short proof for this fact using $d$-collapsibility.

\begin{Lemma}{{\rm [See also \cite[Corollary~3.11]{BYZ}]}}\label{complete}
  The simplicial complex $\langle [n]\rangle$ is chordal. 
  \end{Lemma}

    \begin{proof}
      Let $\Gamma=\langle [n]\rangle$ and observe that $\Gamma$ is a
      $d$-closure. By \cref{d-collapsible=chordal} it is enough
      to show that $\Gamma$ is $d$-collapsible for all $d\geq 1$.  The
      face $E=\{n\}$ appears in the unique facet $[n]$, and is
      therefore a free face of $\Gamma$ of dimension $0<d$ for any $d\geq 1$. We
      now use induction on $n$. If $n=1$, then $\Gamma\coll{E}=\langle
      \emptyset \rangle$ and we are done. If $n>1$, then
      $\Gamma\coll{E}=\langle [n-1]\rangle$ which is $d$-collapsible
      by induction hypothesis, settling our claim.
    \end{proof}

   The condition of being a $d$-closure is necessary in the statement
   of \cref{d-collapsible=chordal}, as can be seen in the
   example below.

   \begin{Example}\label{Example2}
    If $\Gamma=\langle
     \{1,2,3\},\{1,2,4\},\{1,3,4\},\{2,3,4\}\rangle$ is the hollow
     tetrahedron and $d=1$, then $\Delta_1(\Gamma)=\langle
     \{1,2,3,4\}\rangle$ is the full tetrahedron, and so $\Gamma$ is
     $1$-chordal by \cref{complete}. But $\Gamma$ has no free
     face of dimension $< 1$, and therefore $\Gamma$ is not
     $1$-collapsible.
   \end{Example}

   This example shows that $d$-collapsibility of $\Delta_d(\Gamma)$ is
   not a sufficient condition for $\Gamma$ to be $d$-collapsible. It
   is, however, a necessary condition, as we show in
   \cref{d-collapsible skeleton}, which implies, in particular, that
   every $d$-collapsible complex is $d$-chordal, though the converse
   is not true in general (\cref{d-coll=>t-chordal}(a)). To show  \cref{d-collapsible skeleton}, we need the following two lemmas.

\begin{Lemma}\label{special subcomplex is d-coll}
Let $\Sigma$ be a $d$-collapsible simplicial complex on $[n]$ and let $E$ be a subset of $[n]$
with the property that all facets of $\Sigma$
containing $E$ have dimension $\leq d-1$. Then $\Sigma\sm E$ is
$d$-collapsible.
\end{Lemma}

\begin{proof}
If $E=\emptyset$, then $\Sigma\sm E=\emptyset$ which is $d$-collapsible by definition.  Suppose $E\neq \emptyset$.
If $E\notin \Sigma$, then $\Sigma \sm E =\Sigma$ is
$d$-collapsible. Suppose $E \in \Sigma$ and 
$\mathbf{E}=E_1,\ldots,E_t$ is a free sequence of $\Sigma$ with $\dim
E_i\leq d-1$ and $\Sigma\coll{\mathbf{E}}=
\emptyset.$
   
   Suppose $\mathcal{E}=\{F\in \Sigma: E\subseteq F\}$. Then
   $\Sigma\sm E=\Sigma- \mathcal{E}$ and all maximal
   elements of $\mathcal{E}$ are facets of $\Sigma$ with dimension
   $\leq d-1$.

 Suppose $r$ is the smallest element in
   $1,\ldots,t$ with $E_r\subseteq G$ for some $G\in
   \mathcal{E}$. Then $E_1,\ldots,E_{r-1}$ is a free sequence in
   $\Sigma\sm E$, and we have
 \[
 \left(\Sigma\sm E\right)\coll{E_1,\ldots,E_{r-1}}=\left(\Sigma\coll{E_1,\ldots,E_{r-1}}\right)\sm E.
 \]

      So without loss of generality we may
      assume that $r=1$.

      We now proceed with induction on the number of faces of
      $\Sigma$.  If $\Sigma=\emptyset$, then there is nothing to
      prove.  Consider $\Sigma=\langle \{1\}\rangle$ as the base case
      of induction. Then $E=\{1\}$ and $\Sigma\sm E=\emptyset$
      which is $d$-collapsible. Let $G$ be a facet of $\Sigma$ in
      $\mathcal{E}$ containing $E_1$. Then $G$ is a free face of
      $\Sigma$ of dimension $\leq d-1$, and by \cref{l:experimental},
      since $\Sigma \coll{E_1}$ is $d$-collapsible, so is $\Sigma
      \coll{G}$. By induction hypothesis $$\Sigma \coll{G} \sm E
      = \Sigma \sm E$$ is $d$-collapsible, and we are done.
                \end{proof}

\begin{Lemma}\label{a simp induces a sequence}
   Let $E$ be a free face of a simplicial complex $\Gamma$ with
   $\dim E \leq d-1$. Then  there is a free sequence
   $\mathbf{E}=E_1,\ldots,E_r$ of dimension $\leq d-1$ for
   $\Sigma:=\Delta_d(\Gamma)$ such that
 \[
  \Sigma\coll{\mathbf{E}}=
                          \Sigma\sm E.
                         \]
                          \end{Lemma}

    \begin{proof} If $E$ is a free face of $\Sigma$  we
     set $\mathbf{E}=E$ and we are done.

     Suppose that $E$ is not a free face for $\Sigma$ and $E$ is
     contained in a unique facet $F$ of $\Gamma$. Then $\dim E<d-1$,
     because if $\dim E=d-1$, then either $E$ is a facet of $\Gamma$ in
     which case it will be a facet of $\Sigma$, or all $d$-faces of
     $\Gamma$ containing $E$ are contained in $F$, which makes $F$
     also the unique facet of $\Sigma$ containing $E$.

     Let
     $\mathbf{E^1}=\{E^1_1, \ldots,E^1_{m_1}\}$ be  the set of
     $(d-1)$-faces of $\Sigma$ which contain $E$ but are not subsets
     of $F$. Since all $d$-faces of $\Gamma$ containing $E$ are in $F$, there is no
     face of dimension $\geq d$ in $\Sigma$ which contains $E^1_i$ for
     $1\leq i\leq m_1$. This implies that $E^1_i$ is a facet in
     $\Sigma$ for any $i$. So $\mathbf{E^1}$ gives a  free sequence of $\Sigma$.

     If $F$ is the only facet in $\Sigma\coll{\mathbf{E^1}}$
     which contains $E$, then $E$ is a free face and we set
     $$\mathbf{E}= E^1_1, \ldots,E^1_{m_1}, E \mbox{ so that }
     \Sigma\coll{\mathbf{E}}=\Sigma\sm E.$$

     Otherwise, let $\mathbf{E^2}$ be the set of all $(d-2)$-faces of
     $\Sigma\coll{\mathbf{E^1}}$ which contain $E$ but are not subsets
     of $F$. Once again, $\mathbf{E^2}$ is a free sequence in
     $\Sigma\coll{\mathbf{E^1}}$, and continuing in this way after a
     finite number of steps, we get the free sequence
     $$\mathbf{E}=\mathbf{E^1}, \ldots, \mathbf{E^s},E$$ of $\Sigma$
     of dimension $\leq d-1$ for which
     $\Sigma\coll{\mathbf{E}}=\Sigma\sm{E}$.
			    \end{proof}

   \begin{Theorem}[{\bf The $d$-closure of a  $d$-collapsible complex is $d$-collapsible}]\label{d-collapsible skeleton}
     Let $\Gamma$ be a $d$-collapsible simplicial complex for some
     $d\geq 1$. Then $\Delta_{d}(\Gamma)$ is $d$-collapsible. 
\end{Theorem}
			
\begin{proof} Let $\Sigma:= \Delta_{d}(\Gamma)$ and $r$
  be the length of the shortest free sequence of $\Gamma$ which
  $d$-collapses it into $\langle\emptyset\rangle$.  We proceed by
    induction on $r$. If $r=0$ then $\Gamma=\langle \emptyset \rangle$
    and $\Delta_{d}(\Gamma)=\langle [n] \rangle^{[d-1]}$, which is
    $d$-collapsible, because the facets are free faces and all of them have dimension$<d$.

    For the general case, suppose $E$ is the first element in the
    shortest free sequence of length $r$ for $\Gamma$. Then
    $\Gamma\coll{E}$ is $d$-collapsible using a sequence of length
    $r-1$, so by induction hypothesis 
     $\Delta_d(\Gamma\coll{E})$ is
    $d$-collapsible. Now
    $$\Sigma\sm{E}=\Delta_d(\Gamma\coll{E})\sm E = 
      \Delta_d(\Gamma\coll{E})\sm \{G\subset [n] \st E\subseteq G,
    \dim G\leq d-1\}.$$
    
    Since the maximal elements of $\{G\subset [n] \st E\subseteq G,
    \dim G\leq d-1\}$ are facets of $\Delta_d(\Gamma\coll{E})$ and no
    other facet of $\Delta_d(\Gamma\coll{E})$ contains $E$,
    \cref{special subcomplex is d-coll} implies that
    $\Sigma\sm{E}$ is $d$-collapsible.  By \cref{a simp induces a sequence}, there
    is a free sequence $\mathbf{E}=E_1,\ldots,E_t$ of $\Sigma$ of
    dimension $<d$ such that
\begin{align*}
         \Sigma\coll{\mathbf{E}}=
                          \Sigma\sm E
\end{align*}
                        which implies that $\Sigma$ is
                        $d$-collapsible.
	\end{proof}

\Cref{d-collapsible skeleton} has a surprising consequence,
combinatorially, but as we will see later, also algebraically. For the 
experts in clutters,  Part~(b) of the statement below is equivalent to
Theorem~4.3 in Nikseresht's work~\cite[Theorem~4.3]{Nik}. The main 
observation that is behind the statements below is  that if $\Sigma$ is 
$d$-collapsible, then it is $t$-collapsible for all $t\geq d$.

\begin{Proposition}[{\bf $d$-collapsible implies $t$-chordal for
      $t\geq d$}]\label{d-coll=>t-chordal}
     Let $\Gamma$ be a simplicial complex on $[n]$, and let $d\geq 1$. 

      \begin{itemize}

      \item[\rm (a)] If $\Gamma$ is $d$-collapsible then $\Gamma$ is
        $t$-chordal for all $t\geq d$.

      \item[\rm (b)] If $\Gamma$ is $d$-chordal then the simplicial
        complex $\Delta_d(\Gamma)$ is $t$-chordal for all $t\geq d$.

      \item[\rm(c)] If $\Gamma$ is a $d$-closure then $\Gamma$ is
        chordal if and only if $\Gamma$ is $d$-chordal.

      \end{itemize}

\end{Proposition}

\begin{proof}
  (a) By the definition of $d$-collapsibility, we know that $\Gamma$
  is $t$-collapsible for all $t\geq d$. Using \cref{d-collapsible
    skeleton}, for all $t\geq d$, $\Delta_t(\Gamma)$ is
  $t$-collapsible. Hence by \cref{d-collapsible=chordal}, for all
  $t\geq d$, $\Delta_t(\Gamma)$ is $t$-chordal.  Therefore $\Gamma$ is
  $t$-chordal for $t\geq d$.

  (b) Since $\Gamma$ is $d$-chordal for some $d\geq 1$, by
  \cref{d-collapsible=chordal} the simplicial complex
  $\Sigma=\Delta_d(\Gamma)$ is $d$-collapsible. Hence $\Sigma$ is
  $t$-collapsible for all $t\geq d$. \Cref{d-collapsible skeleton}
  implies that $\Delta_t(\Sigma)$ is $t$-collapsible, and by
  \cref{d-collapsible=chordal}, $\Sigma$ is $t$-chordal for all $t\geq
  d$.
  
  (c) Suppose $\Gamma$ is $d$-chordal, and let $t<d$. Since
  $\Gamma$ contains all $(t+1)$-subsets of $[n]$ we have
  $\Delta_t(\Gamma)=\langle [n]\rangle$ which by \cref{complete} is
  $t$-chordal. For $t\geq d$, since $\Gamma=\Delta_d(\Gamma)$ Part~(b)
  implies the assertion.
\end{proof}

Note that in the assumption of \cref{d-coll=>t-chordal}(a), we cannot replace
$d$-collapsibility of $\Gamma$ with $d$-collapsibility of
$\Delta_d(\Gamma)$ (or, equivalently, by  $d$-chordality of $\Gamma$). Let
$\Gamma$ be a simplicial complex which is not chordal and let $d\geq
1$ be an integer with $d<r$, where $r$ is the smallest dimension of
the nonfaces of $\Gamma$. Then $\Delta_d(\Gamma)=\langle [n]\rangle$,
which is $d$-chordal and hence by \cref{d-collapsible=chordal} it is
$d$-collapsible. But since $\Gamma$ is not chordal, there exists
$t\geq r>d$ such that $\Gamma$ is not $t$-chordal.

We now examine the relation between free faces of dimension $d-1$ in a
simplicial complex $\Gamma$ and its $d$-closure.  We saw in \cref{Example2}
  that it is possible for $\Gamma$ to be
  $d$-chordal without having a free face of dimension $d-1$. In other
  words $\Delta_d(\Gamma)$ having a free face of dimension $d-1$ does
  not imply that $\Gamma$ has a free face of dimension $d-1$. But the
  converse is true.

      \begin{Proposition}\label{free sequence=>simp sequence} 
         Let $\Gamma$ be a simplicial complex on the vertex set $[n]$,
         and let $\mathbf{E}=E_1,\ldots,E_t$ be a sequence of
         $(d-1)$-faces of $\Gamma$ with the property that $E_1$ is a free face in
         $\Gamma$ and $E_i$ is a free face in
         $\Gamma\sr{E_1,\ldots,E_{i-1}}$ for $i>1$. Then
\begin{itemize}

\item[{\rm(a)}]   $\mathbf{E}$ is a simplicial sequence for $\Delta_d(\Gamma)$;
\item[{\rm(b)}] If $\Gamma\sr{\mathbf{E}}\subseteq\langle
  [n]\rangle^{[d-1]}$, then $\mathbf{E}$ contains  a simplicial order for
  $\Delta_d(\Gamma)$.
\end{itemize}
\end{Proposition}

\begin{proof}
  (a) Let $F$ be the unique facet in $\Gamma$ containing $E_1$.  Since
  $\Gamma\subseteq \Delta_d(\Gamma)$ we have $F\in
  \Delta_d(\Gamma)$. Suppose $E_1\cup \{v\}\in \Delta_d(\Gamma)$ for
  some $v\in [n]-E$. Then $E_1\cup \{v\}\in \Gamma$ because any
  $d$-face of $\Delta_d(\Gamma)$ belongs to $\Gamma$. It follows that
  $v\in F$. Hence $F$ is the only facet containing $E_1$ in
  $\Delta_d(\Gamma)$. Thus $E_1$ is simplicial in $\Delta_d(\Gamma)$,
  and by \cref{l:easy} $\mathbf{E}$ is a simplicial sequence for
  $\Delta_d(\Gamma)$.

\medspace
  (b) Let $\bf{E'}$ be the subsequence of $\bf{E}$ which contains all
  elements in $\bf{E}$ which are not facets in
  $\Delta_d(\Gamma)$. Note that if $E_i$ in $\bf{E}$ is a facet of
  $\Delta_d(\Gamma)$, then it is a facet in $\Gamma$ too. Hence
  $\Gamma\sr{E_1,\ldots,E_{i}}=\Gamma\sr{E_1,\ldots,E_{i-1}}$. Now,
  Part $(a)$ implies that $\bf{E'}$ is a simplicial sequence for
  $\Delta_d(\Gamma)$ and the assertion follows from the fact that
  $\Delta_d(\Gamma)=\langle [n]\rangle^{[d-1]}$, whenever
  $\Gamma\subseteq \langle [n]\rangle^{[d-1]}$.
\end{proof}

  As promised earlier, in \cref{the bounds for checking chordality} we show that to
  check chordality, it is enough to check $d$-chordality for a finite
  number of positive integers $d$. Note that \cref{d-coll=>t-chordal}(c) can be
  also deduced from \cref{the bounds for checking chordality}.

\begin{Proposition}\label{the bounds for checking chordality}
  Let $\Gamma$ be a simplicial complex with vertex set $[n]$ and
  dimension $r$, let $$t=\min\{\dim F \st F \subseteq [n],\ \text{a
    minimal nonface of } \Gamma\}$$ and $$s=\max\{\dim F \st
  F\subseteq [n],\ \text{a minimal nonface of } \Gamma\}.$$ The
  following conditions are equivalent.
  \begin{itemize}
  \item[\rm{(i)}] $\Gamma$ is chordal;
  \item[\rm{(ii)}] $\Gamma$ is $d$-chordal for $t\leq d\leq \min\{r, s\}$.
  \end{itemize}
\end{Proposition}

\begin{proof} The implication $(i)\implies (ii)$ follows from the
  definition of chordality.

  For $(ii)\implies (i)$, note that if $d<t$, then since all $F\subset
  [n]$ with $\dim F=d$, are in $\Gamma$, we have
  $\Delta_{d}(\Gamma)=\langle [n]\rangle$, which is $d$-chordal by
  \cref{complete}. So $\Gamma$ is $d$-chordal for $d<t$.

  If $d> r$, then there is no $d$-face in $\Gamma$ and so we will
  automatically have $\Delta_{d}(\Gamma)=\langle [n]\rangle^{[d-1]}$,
  which satisfies condition $(*)$ and is therefore $d$-chordal.

Now let  $\min\{r,s\}=s$ and $d>s$. We claim that 
\begin{align}\label{taking  closures twice}
\Delta_d(\Gamma)=\Delta_d(\Delta_s(\Gamma)).
\end{align}
Then since by assumption $\Delta_s(\Gamma)$ is $s$-chordal, it follows from  \cref{d-coll=>t-chordal}(b) that $\Delta_d(\Gamma)$  is $d$-chordal for $d\geq s$. This implies that $\Gamma$ is chordal, as desired.

Next we prove the claim. To do this we show that for $d\geq s$ 
\begin{align}
\label{first}\Delta_d(\Gamma)&=\langle [n]\rangle^{[d-1]}\cup \Delta_s(\Gamma),\text{ and}\\
\label{second}\Delta_d(\Delta_s(\Gamma))&=\langle [n]\rangle^{[d-1]}\cup \Delta_s(\Gamma).
\end{align}

To prove \cref{first}, we first observe that by definition $\langle
[n]\rangle^{[d-1]}\subseteq \Delta_d(\Gamma)$, and so we need to only
worry about faces of dimension $\geq d$. Let $F \subseteq [n]$ with
$\dim F \geq d$.

If $F\in \Delta_d(\Gamma)$ then any $d$-face of $F$ belongs to
$\Gamma$, and since $d\geq s$, it follows that any $s$-face of $F$ is
in $\Gamma$. Hence $F\in \Delta_s(\Gamma)$. 

If $F\in \Delta_s(\Gamma)$  and $F\notin
\Delta_d(\Gamma)$,  then there is $G\subseteq F$ with $\dim G=d$ and
$G\notin \Gamma$. Since the minimal nonfaces of $\Gamma$ have
dimension $\leq s$, there exists $H\subseteq G$, where $H \notin
\Gamma$ and $\dim H=s$. But $F\in \Delta_s(\Gamma)$, and any $s$-face
of $F$ belongs to $\Gamma$, hence $H\in \Gamma$, a contradiction. So
$F\in \Delta_d(\Gamma)$. This settles \cref{first}.

Now we prove \cref{second}. The containment $``\supseteq"$ holds by
definition. Suppose $F\in \Delta_d(\Delta_s(\Gamma))$. If $\dim F<d$,
then $F\in \langle [n]\rangle^{[d-1]}$. If $\dim F\geq d$, then any
$d$-face $G$ of $F$ belongs to $\Delta_s(\Gamma)$. Hence any $s$-face
$H$ of $G$ (and hence $F$) is in $\Gamma$. It follows that $F\in
\Delta_s(\Gamma)$, as desired. This settles \cref{second}.

\Cref{taking  closures twice} now follows, and the proof is complete.
\end{proof}

The following theorem shows that to check the chordality of
a simplicial complex, it is enough to check its $d$-collapsibility for
 one appropriate $d$.

\begin{Theorem}\label{d-collapsible for suitable d}
  If $\Gamma$ is a $d$-collapsible simplicial complex and $\dim
  F\geq d \geq 1$ for all nonfaces $F$ of $\Gamma$,   then $\Gamma$ is
  chordal.
\end{Theorem}

\begin{proof} By assumption $d \leq r=\min\{\dim F \st F \text{ a
    nonface of } \Gamma\}$.  By \cref{d-coll=>t-chordal}(a), $\Gamma$ is
  $t$-chordal for all $t\geq d$, and in particular for all $t\geq r$,
  and so from \cref{the bounds for checking chordality}, $\Gamma$ is chordal.
\end{proof}

\begin{Example}\label{Example3} We continue with $\Gamma$ as in 
  \cref{Example0}. 
  Consider
  $$\Delta_1(\Gamma)=\langle \{1,2,4,5\},\{1,2,3,4\}\rangle$$ 
  calculated in \cref{Example0}. Then $E_1=\{5\}$ is contained in
  only one facet and hence is simplicial. So
$$\Delta_1(\Gamma)\sr{E_1}=\langle \{1,2,3,4\}\rangle\cup\langle \{5\}\rangle.$$

In order to see $\Gamma$ is $1$-chordal, now  it is enough to find a simplicial order for
$\langle \{1,2,3,4\}\rangle$. But it follows from \cref{complete}
that $\langle \{1,2,3,4\}\rangle$  admits a simplicial order. 

The work done in \cref{Example1} shows that
  $\Gamma$ is $2$-chordal. 

  Since $\max\{\dim F \st F\subseteq [n],\ \text{a minimal nonface of
  } \Gamma\}=2$, it follows from \cref{the bounds for checking chordality} that $\Gamma$ is
  chordal.
 
Note that $\Gamma$ is not $1$-collapsible and hence we cannot make use of \cref{d-collapsible for suitable d} to prove that $\Gamma$ is chordal.
\end{Example}

One can even see that the induced subcomplexes of a simplicial complex
inherit chordality.

\begin{Proposition}[\bf{Chordality of induced subcomplexes}]\label{induced complexes of chordals}
	Let $\Gamma$ be a simplicial complex on the vertex set $[n]$,
        $d$ be a positive integer and let $W\subset [n]$.
	\begin{itemize}
        \item[\rm(a)] For a face $E$ of $\Gamma$ one has
          $(\Gamma\sm E)_W=(\Gamma_W)\sm E$.
        \item[{\rm (b)}] As simplicial complexes on the vertex set
          $W$ we have $\Delta_d(\Gamma)_W=\Delta_d(\Gamma_W)$, and in
          particular if $\Gamma$ is a $d$-closure on $[n]$, then so is
          $\Gamma_W$ on $W$.
        \item[{\rm (c)}] If $E\subseteq W$ is a free face of $\Gamma$
          then $E$ is a free face of $\Gamma_W$.

        \item[{\rm (d)}] If $E\subseteq W$ with $E\in
          \Simp(\Delta_d(\Gamma))$, then $E\in
          \Simp(\Delta_d(\Gamma_W))$.
        \item[{\rm (e)}] If $\Gamma$ is $d$-chordal then $\Gamma_W$ is
          $d$-chordal.

	\item[{\rm (f)}] If $\Gamma$  is chordal then $\Gamma_W$ is chordal.
	\end{itemize}
\end{Proposition}

\begin{proof}
(a) We have
\begin{align*}
F\in (\Gamma\sm E)_W
&\Leftrightarrow  F\subseteq W\text{ and } F\in \Gamma\sm E\\
&\Leftrightarrow F\subseteq W,\ F\in \Gamma\text{ and }E\not\subseteq F\\
&\Leftrightarrow F\in \Gamma_W\text{  and }E\not\subseteq F\\
&\Leftrightarrow F\in (\Gamma_W)\sm E.
\end{align*}

\medspace (b) By definition of $d$-closure both simplicial complexes
$\Delta_d(\Gamma)_W$ and $\Delta_d(\Gamma_W)$ contain all subsets of
$W$ of cardinality~$\leq d$. Let $F \subseteq W$ with $|F|>d$. Then by
definition of $d$-closures
\begin{center}$F\in \Delta_d(\Gamma)_W$ 
$\Leftrightarrow$  all $d$-faces of $F$ are in $\Gamma$
 $\Leftrightarrow$ all $d$-faces of $F$ are in $\Gamma_W$
$\Leftrightarrow$ $F\in \Delta_d(\Gamma_W)$ 
\end{center}
which settles our claim.

 \medspace (c)  Suppose $E$
  is contained in the unique facet $F$ of $\Gamma$.  Since the
  facets of $\Gamma_W$ are the maximal elements of $\{G\cap W \st G\in
  \facets(\Sigma)\}$, we see that $E$ is contained in the unique facet
  $F\cap W$ of $\Gamma_W$. Hence $E$ is a free face of $\Gamma_W$.
  
\medspace (d)  Follows from Part~(b) and Part~(c).

\medspace (e) Since $\Gamma$ is $d$-chordal,
\cref{d-collapsible=chordal} implies that $\Delta_d(\Gamma)$ is
$d$-collapsible. It follows from \cite[Lemma~2]{We}  that
$\Delta_d(\Gamma_W)=\Delta_d(\Gamma)_W$ is  $d$-collapsible, and  hence by
\cref{d-collapsible=chordal} $\Gamma_W$ is $d$-chordal. 

\medspace (f) Since $\Gamma$ is chordal, it is $d$-chordal for
all $d\geq 1$. Part~(e) implies the assertion.

\end{proof}


\subsection*{$d$-representable complexes}

				Let $A = \{A_1,A_2, \ldots ,A_n\}$ be a family of sets.  Consider the following family of subsets of $A$
			\[ N(A) := \{F \subset [n] :  \ \cap_{i\in F}A_i\neq \emptyset\}.
			\]
This finite family is a simplicial complex which is called the {\bf
  Nerve Complex} of $A$. A simplicial complex which is the nerve
complex of some finite family of convex sets in $\mathbb{R}^d$ is
called {\bf$d$-representable}. One of the main problems regarding
nerve complexes is to characterize $d$-representable complexes. This
problem is solved in case $d=1$, see \cite{LB}. For $d>1$ the problem
is still open. The reader may consult with \cite{We} for more
information about $d$-representable complexes.

\begin{Theorem}[{\bf $d$-representable complexes are $d$-chordal}]\label{representable}
Let $\Gamma$ be a $d$-representable simplicial complex on the vertex
set $[n]$. Then $\Delta_d(\Gamma)=\Gamma\cup \langle
[n]\rangle^{[d-1]}$. Moreover, $\Gamma$ is $d$-chordal.
\end{Theorem}                      
\begin{proof} The inclusion $\Gamma\cup \langle [n]\rangle^{[d-1]}\subseteq
       \Delta_d(\Gamma)$ always holds. For the converse, we use a
       celebrated theorem of Helly, \cite{He}, which states that if
       each $d+1$ members of a finite family of at least $d+1$ convex
       sets in $\mathbb{R}^d$ have nonempty intersection, then the
       whole family intersects. This implies that if $F\subset [n]$,
       $|F| \geq d+1$ and each $(d+1)$-subset of $F$ belongs to
       $\Gamma$, then $F\in\Gamma$. Hence any $t$-face of
       $\Delta_d(\Gamma)$ with $t\geq d$ is a face in $\Gamma$. It
       follows that $\Delta_d(\Gamma) \subseteq \Gamma\cup \langle
       [n]\rangle ^{[d-1]}$. This proves the equality.
			
       Wegner~\cite{We} proved that $d$-representable complexes are
       $d$-collapsible, so $\Gamma$ is $d$-collapsible, and by
       \cref{d-collapsible skeleton}, $\Delta_d(\Gamma)$ is
       $d$-collapsible as well. Now
       \cref{d-collapsible=chordal} yields the result.
			\end{proof}

      \begin{Remark} The converse of \cref{representable} does not hold in
      general. Let $\Gamma$ be a simplicial complex of dimension $<d$
      which is not $d$-representable, for example the complex $C_2$ in
      \cite[Figure~2]{We}. Then $\Delta_d(\Gamma)=\langle
           [n]\rangle^{[d-1]}$ and hence $\Gamma$ is $d$-chordal by definition.

           The converse of \cref{representable} is not true even for
           $d$-closures: there are simple examples of $d$-closures
           which are $d$-chordal but not $d$-representable:
           \cref{star} illustrates a chordal graph which can be
           viewed as a $1$-closure simplicial complex (and hence
           $1$-chordal). But since it is not an interval graph it is
           not $1$-representable.

      \begin{figure}[ht!]
					\begin{center}
		\begin{tikzpicture}[line cap=round,line join=round,>=triangle 45,x=0.7cm,y=0.7cm]
\clip(2.,1.) rectangle (8.,6.);
\draw [color=black][line width=0.8pt] (5.,5.)-- (5.,3.);
\draw [color=black][line width=0.8pt] (5.,3.)-- (3.,2.);
\draw [color=black][line width=0.8pt] (5.,3.)-- (7.,2.);
\draw (4.72,2.97) node[anchor=north west] {\begin{scriptsize}$1$\end{scriptsize}};
\draw (2.52,2.22) node[anchor=north west] {\begin{scriptsize}$2$\end{scriptsize}};
\draw (6.9,2.22) node[anchor=north west] {\begin{scriptsize}$3$\end{scriptsize}};
\draw (4.72,5.7) node[anchor=north west] {\begin{scriptsize}$4$\end{scriptsize}};
\begin{scriptsize}
\draw [fill=black] (5.,5.) circle (1.5pt);
\draw [fill=black]  (5.,3.) circle (1.5pt);
\draw [fill=black]  (3.,2.) circle (1.5pt);
\draw [fill=black]  (7.,2.) circle (1.5pt);
\end{scriptsize}
\end{tikzpicture}
\caption{A $1$-chordal $1$-closure which  is not $1$-representable}\label{star}
		\end{center}
		\end{figure}
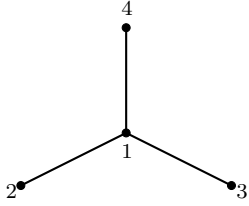

       Also, it is possible for a simplicial complex to not be
       $d$-representable, while its $d$-closure is
       $d$-representable. For example, $\langle [n]\rangle^{[d-1]}$
       which is $d$-closure of all simplicial complexes of dimension
       $<d$ is $d$-representable: Let $A=\{A_1,\ldots, A_n\}$, where
       $A_i$ are affine hyperplanes in $\mathbb{R}^d$. Then any $d$ of
       them intersect in a point, and no $d+1$ of them intersect, as
       the dimension of the intersection reduces by one each time we
       intersect with a new affine hyperplane. Hence $\langle
       [n]\rangle^{[d-1]}=N(A)$.
\end{Remark}

 
 \section{Applications to monomial ideals}
 We now apply the combinatorial results in the previous
   sections to minimal free resolutions of monomial ideals. Let $I$ be
   a square-free monomial ideal in the polynomial ring
   $K[x_1,\ldots,x_n]$ over a field $K$, with Stanley-Reisner complex
   $\Gamma$.  We write $I_{\langle j\rangle}$ for the ideal generated
   by all homogeneous polynomials of degree $j$ belonging to $I$. We
   say that $I$ is {\bf componentwise linear}~\cite{HHCWL} if
   $I_{\langle j\rangle}$ has a linear resolution for all $j$.
   Componentwise linear ideals generalize ideals with linear
   resolution, in the sense that an ideal with linear resolution is
   componentwise linear:  if $I$ is generated in a fixed
   degree $d$ and has linear resolution, then all $I_{\langle k \rangle}$ have
   linear resolutions. This is the perspective we take when
   chordality is being considered; see~\cref{the bounds for checking chordality}.

   If $I$ is a square-free monomial ideal, then by $I_{[j]}$ we mean
   the square-free monomial ideal generated by all the square-free
   monomials of degree $j$ belonging to $I$. The ideal $I$ is called {\bf square-free componentwise linear} if $I_{[j]}$ has a linear resolution for all $j$. Herzog and
   Hibi~\cite{HHCWL} proved that a square-free monomial ideal  is componentwise linear if and
   only if it is square-free componentwise linear.

   For $E \subseteq [n]$, we set $${\bfx}_E=\prod_{i\in
     E}x_i.$$

   The main tool used in  this section is examining, for a free face $E$ of
   $\Gamma$, how adding $\bfx_E$ to the generating set of $I$ affects
   the Betti numbers of $I$. As a consequence, among other things, we
   are able to produce large classes of componentwise linear ideals.

   We begin with some basic observations.

 \begin{Lemma}\label{l:prep-lemma}  Let $I$ be  a square-free monomial ideal
              in $K[x_1,\ldots,x_n]$, $K$ a field, and let
              $\Gamma=\N(I)$.

   \begin{itemize}
	\item[{\rm(a)}] $\N(I_{[d+1]})=\Delta_d(\Gamma)$ for all $d$. 

        \item[{\rm(b)}] If  $E \subseteq [n]$, then 
           $\N(I+({\bfx}_E))=\Gamma\sm {E}.$

        \item[{\rm(c)}] If $E$ is a free face of $\Gamma$, then
              $\N(I+({\bfx}_E))=\Gamma\coll{E}.$ 
   \end{itemize}
 \end{Lemma}

\begin{proof}

  (a) First note that both $\N(I_{[d+1]})$ and $\Delta_d(\Gamma)$
  contain all possible faces of dimension $< d$. Suppose $F \subseteq
  [n]$ and $|F| \geq d+1$. Then
 \begin{align*}
F \in \N(I_{[d+1]}) & 
\iff \bfx_F \notin I_{[d+1]}\\ 
& \iff \forall G \subseteq F \mbox{ with } |G|=d+1,  \bfx_G \notin I_{[d+1]}\\ 
& \iff \forall G \subseteq F \mbox{ with } |G|=d+1,  G \in \Gamma\\
& \iff F \in \Delta_d(\Gamma).
\end{align*}

  (b) If $\sigma \subseteq [n]$, then $$\sigma \in \N(I+{\bfx}_E) \iff
  {\bfx}_\sigma \notin (I + {\bfx}_E)
  \iff {\bfx}_\sigma \notin I \mbox{ and } {\bfx}_E \nmid {\bfx}_\sigma
  \iff \sigma \in \Gamma\sm {E}.$$

  (c) This statement follows directly from Part~(b).
  \end{proof}

We now turn to the effect of the operation of $d$-collapsing on the
reduced homology modules of a simplicial complex. It is well known
that simplicial collapsing preserves reduced homology modules (see for
example~\cite[Theorem~6.6, Definition~6.13 and
  Proposition~6.14]{K}). In the special case of $d$-collapsing this is
true only for higher reduced homology modules, since we allow facets
as free faces.

We write a proof for this fact, since we could not find one in the
literature, but it is folklore (see also \cite[Proposition~2.3]{BiY}).

\begin{Proposition}\label{p:d-collapsing-hom} 
  If $\Gamma$ is a simplicial complex with a free face $E$,
  then $$\wh_i(\Gamma;K) \cong
  \wh_i(\Gamma\coll{E};K) 
   \ \ \ \ \mbox{ for } \ \ \ \ 
   i > \left \{ 
  \begin{array}{ll}
  \dim E & E \in \facets(\Gamma)\\
  0 & E \notin \facets(\Gamma).
  \end{array}\right.$$
\end{Proposition}

\begin{proof}
  This follows from a simple application of the Mayer-Vietoris
  sequence: if $F$ is the unique facet in $\Gamma$ containing $E$,
  then $\langle F\rangle=\langle E\rangle*\langle G\rangle$, where the
  operation $*$ denotes simplicial join and $G=F-E$. Then, setting
  $\Gamma'=\Gamma\coll{E}$, we have:
$$\Gamma' \cup \langle F\rangle=\Gamma \mbox{ and } \Gamma' \cap
\langle F\rangle=\partial(E)*\langle G\rangle,$$ where $\partial(E)$
is the boundary complex of $E$. 
The Mayer-Vietoris sequence
(e.g.~\cite[Theorem~5.17]{K}) gives 
\begin{align}\label{MV}
  \cdots\!\to 
  \wh_i(\partial(E)\!*\!\langle G\rangle;K)\!\to\!
  \wh_i(\Gamma';K)\!\oplus\!\wh_i(\langle F\rangle;K)\!\to\!\wh_i(\Gamma;K)\!\to\!\wh_{i-1}(\partial(E)\!*\!\langle G\rangle;K) \to\!
  \cdots.
\end{align}

Note that $\wh_i(\langle F\rangle;K)=0$ for all $i$. If $G\neq \emptyset$, then $\partial(E)*\langle G\rangle$ is a cone and
hence acyclic, and (\ref{MV}) gives the isomorphism of the homology
modules for all $i > 0$ (this is the better-known case of an elementary
collapse). If $G=\emptyset$ (this is the case when $E$ is a facet of
$\Gamma$), then the same argument gives us the isomorphism of the
homology modules for $i > \dim E$.
\end{proof}

 Our statement about Betti numbers in \cref{betti numbers
    of chordals} is a generalization of \cite[Theorem~2.1]{BYZ}. For
  the proof we will need the following statement form \cite{HHCWL}.

\begin{Lemma}[{\cite[Lemma~1.2]{HHCWL}}]\label{betti numbers of ideals of deg<k}
Let $I\subset S$ be a graded ideal, and for a nonnegative integer $k$
let $I_{\leq k}$ denote the ideal generated by all homogeneous
polynomials of $I$ whose degree is less than or equal to $k$.  Then
for all $k$ and all $j\leq k$ we have
 $$\beta_{i,i+j}(I)=\beta_{i,i+j}(I_{\leq k}).$$
\end{Lemma}


\begin{Theorem}[\bf Betti numbers from free faces]\label{betti numbers of chordals}
  Let $I$ be a square-free monomial ideal of $S=K[x_1,\ldots,x_n]$
  where $K$ is a field, $\Gamma=\N(I)$, and  $E \subseteq [n]$ with $|E|=d$.
  \begin{itemize}

  \item[\rm{(a)}] If $E$ is a free face of $\Gamma$,
    then for every $i$
    $$\beta_{i,i+j}(I+({\bfx}_E))=\beta_{i,i+j}(I) \mbox{ for all } j\neq d, d+1.$$

 Moreover, if $E \notin \facets(\Gamma_W)$ for every $W
  \subseteq [n]$ with $|W|=a>2$, then
 $$\beta_{i,i+j}(I+({\bfx}_E))=\beta_{i,i+j}(I)$$ for all $i$ and all
 $j>2$ such that $i+j=a$.

  \item[\rm{(b)}]  If $E$ is a free face of $\Gamma$ or
    of $\Delta_d(\Gamma)$ and
     $$d+1\geq\max \{\deg u \st u \text{ a minimal
    generator of } I\}$$
     and $A\subseteq [n]-E$ such that $E\cup
    \{m\}\in \Gamma$ for all $m\in A$, then for every $i$ we
    have
   \begin{align}\label{main equality}
   \beta_{i,i+j}(I+(x_m{\bfx}_E \st m\in A))= \beta_{i,i+j}(I)
    \mbox{ for all } j\neq d+1.
    \end{align}
   
    Moreover, if $I$ is minimally generated by monomials of degree
    $d+1$, and
    $$E \notin \facets(\left(\N(I+(x_m{\bfx}_E \st m\in
    A))\right)_W)$$ for every $W \subseteq [n]$ with $|W|=a>2$, then
  \begin{align}\label{for j>2}
  \beta_{i,i+j}(I+({\bfx}_E))= \beta_{i,i+j}(I+(x_m{\bfx}_E \st m \in
A))= \beta_{i,i+j}(I)
\end{align}
for all $i$ and all $j>2$ such that $i+j=a$.

  \end{itemize}
\end{Theorem}

\begin{proof}
  (a) By Hochster's formula~\cite{Hochster}(See also \cite[Theorem~8.1.1]{HHBook}) and
  \cref{l:prep-lemma}(c), for all $i$ and $j$
\begin{align*}
\beta_{i,i+j}(I+({\bfx}_E))
&= \sum_{W\subseteq [n]\atop |W|=i+j}\dim_K \wh_{j-2}(\N(I+({\bfx}_E))_W;K)\\
&= \sum_{W\subseteq [n]\atop |W|=i+j}\dim_K \wh_{j-2}((\Gamma \coll{E})_W;K).
\end{align*}

If $E\not\subseteq W$, then by abuse of notation
$\Gamma_W\coll{E}=\Gamma_W$.  If $E\subseteq W$, then by
\cref{induced complexes of chordals} Parts~(c) and (a), $E$ is a free face of $\Gamma_W$,
and $(\Gamma\coll{E})_W=(\Gamma_W)\coll{E}$. By
\cref{p:d-collapsing-hom} 
\begin{align}\label{for-moreover}
\beta_{i,i+j}(I+({\bfx}_E))
=\sum_{W\subseteq [n]\atop |W|=i+j}\dim_K \wh_{j-2}(\Gamma_W;K)
=\beta_{i,i+j}(I) 
\end{align}
for $j>d+1$.

Since $\deg {\bf x}_E=d$ we have $I_{\leq d-1}=(I+({\bfx}_E))_{\leq d-1}$. Hence  by \cref{betti numbers of ideals of deg<k} 
\begin{align*}
\beta_{i,i+j}(I)=\beta_{i,i+j}(I_{\leq d-1})=\beta_{i,i+j}((I+({\bfx}_E))_{\leq d-1})=\beta_{i,i+j}(I+({\bfx}_E))
\end{align*}
for all $i$ and all $j\leq d-1$. 

 Moreover, if for every $W$ with $|W|=a>2$ we have $E
  \notin \facets(\Gamma_W)$, then using \cref{for-moreover} and by
  \cref{p:d-collapsing-hom}
 $$\beta_{i,i+j}(I+({\bfx}_E))=\beta_{i,i+j}(I)$$ for all $i$ and all
 $j>2$ such that $i+j=a$.

\medspace
\noindent (b) We first deal with the case where  all generators of
$I$ have degree $d+1$. Let $$\Gamma=\Delta_d(\Gamma)=\langle
G_1,\ldots,G_t\rangle$$ where $G_t$ is the unique facet of $\Gamma$
containing $E$, and let
$$\Sigma=\Gamma\sm F_1\sm \cdots\sm F_{r},$$
 where $F_k=E\cup\{m_k\}$ for each $m_k\in A$. 
 Now
$$ \Sigma=\langle G_1,\ldots,G_{t-1}\rangle\cup \langle G_t- \{i\} 
: \ i\in E\rangle\cup \langle G_t-A\rangle. $$

It follows that $E$ is uniquely contained in $G_t- A$, and hence it is a
free face of $\Sigma$.  Applying Part~(a) to $\Sigma$,  we see that
\begin{align}\label{sigma1}
\beta_{i,i+j}(\N(\Sigma)) = 
\beta_{i,i+j}(\N(\Sigma)+(\mathbf{x}_E))  
\mbox{ for all } i \mbox{ and all } 
     j> d+1.
\end{align}

On the other hand, $\Sigma\coll{E}=\Gamma\coll{E}$. Therefore, 

\begin{align}\label{sigma2}
\N(\Sigma)+(\mathbf{x}_E)=\N(\Sigma\coll{E})=\N(\Gamma\coll{E})=I+(\mathbf{x}_E).
\end{align}

This implies that for all $i$ and all 
$j> d+1$,

\begin{align*}
\beta_{i,i+j}(I)&=\beta_{i,i+j}(I+(\mathbf{x}_E))&\text{(using Part (a))}\\
&=\beta_{i,i+j}(\N(\Sigma)+(\mathbf{x}_E))&\text{(using \cref{sigma2})}\\
&=\beta_{i,i+j}(\N(\Sigma))&\text{(using \cref{sigma1})}\\
&=\beta_{i,i+j}(I+(x_m{\bfx}_E: m\in A))&\text{(using \cref{l:prep-lemma}(b))}.\\
\end{align*}

 Since the ideals $I$ and $I+(x_m{\bfx}_E \st m\in A)$ are
  minimally generated in degree $d+1$ they both have Betti numbers
  equal to $0$ when $j\leq d$.

  For $W \subseteq [n]$ if we have $E \notin
\facets(\Sigma_W)$, then $E \notin \facets(\Gamma_W)$. Hence if for every $W$ with $|W|=a>2$ we have $E\notin \facets(\Gamma_W)$, then 
 $$\beta_{i,i+j}(I+({\bfx}_E))=
    \beta_{i,i+j}(I+(x_m{\bfx}_E \st m\in A))=
    \beta_{i,i+j}(I)$$
for all $i$ and all
 $j>2$ such that $i+j=a$.

    This settles the equigenerated case. Now suppose all generators of
    $I$ have degree $\leq d+1$. By \cref{l:prep-lemma}(a),
    $\Delta_d(\Gamma)=\N(I_{[d+1]})$.  Observe that if $E$ is free in
    $\Gamma$, then it is also a free face of $\Delta_d(\Gamma)$,
    otherwise it would be contained in at least two facets of $\Gamma$
    which would contradict it being free.  By our discussions above
\begin{align}\label{betti of component}
\beta_{i,i+j}(I_{[d+1]})=\beta_{i,i+j}(I_{[d+1]}+(\mathbf{x}_E))=\beta_{i,i+j}(I_{[d+1]}+(x_m\mathbf{x}_E \st m\in A)),
\end{align}
 for all $i$ and all 
 $j\neq d,d+1$.

Set
$t=\max\{\deg u \st u \text{ minimal generator of } I\}$. It is
 proved in~\cite[Lemma~4.2]{BiY} that if $d+1\geq t$, then
for all $i$ and all $j>d+1$
\begin{align}\label{betti of component=betti of ideal}
\beta_{i,i+j}(I)=\beta_{i,i+j}(I_{[d+1]}).
\end{align}
 It follows from \cref{betti of component} and \cref{betti of component=betti of ideal} that for all $i$ and all $j>d+1$
\begin{align}\label{betti of extension=betti of ideal}
\beta_{i,i+j}(I)=\beta_{i,i+j}(I_{[d+1]}+(x_m\mathbf{x}_E \st m\in [n]-E)).
\end{align}

Let $J:=I+(\mathbf{x}_E)$. Then
$$J_{[d+1]}=I_{[d+1]}+(x_m\mathbf{x}_E \st m\in [n]-E).$$ So by \cref{betti of extension=betti of ideal} for all $i$
and all $j>d+1$
\begin{align}\label{betti of I=betti of J's component}
\beta_{i,i+j}(I)=\beta_{i,i+j}(J_{[d+1]}).
\end{align}
Now set $s=\max\{\deg u \st u \text{ minimal generator of } J\}$.  Since $\deg {\bf x}_E=d$ we have $t\geq s$ and hence 
$d+1\geq s$. Again \cite[Lemma~4.2]{BiY} implies that for all $i$ and
all $j>d+1$
\begin{align}\label{betti of J=betti of J's component}
\beta_{i,i+j}(J)=\beta_{i,i+j}(J_{[d+1]}).
\end{align}

 \Cref{betti of I=betti of J's component} and \cref{betti of J=betti of J's component} yield the following result
\begin{align}\label{number 1}
  \beta_{i,i+j}(I)=\beta_{i,i+j}(I+(\mathbf{x}_E)) \quad\quad\text{for
    all $i$ and all $j>d+1$}.
\end{align}

 By \cref{betti numbers of ideals of deg<k},
\cref{number 1} also holds for $j\leq d-1$.

Now let $L=I+(x_m\mathbf{x}_E \st m\in A)$ for $A\subseteq [n]-E$. Then
$$L_{[d+1]}=I_{[d+1]}+(x_m\mathbf{x}_E \st m\in A).$$ Using \cref{betti of component}
and \cref{betti of component=betti of ideal} one has
\begin{align}\label{number 2}
\beta_{i,i+j}(L_{[d+1]})=\beta_{i,i+j}(I) \quad\quad\text{for all $i$ and all $j>d+1$}.
\end{align}
Since $d+1\geq \max\{\deg u \st u \text{ minimal generator of } L\}$, by \cite[Lemma~4.2]{BiY}  we have 
\begin{align}\label{number 3}
\beta_{i,i+j}(L)=\beta_{i,i+j}(L_{[d+1]}) \quad\quad\text{for all $i$ and all $j>d+1$}.
\end{align}
Consequently, using \cref{number 2} and \cref{number 3}
\begin{align}\label{number 4}
\beta_{i,i+j}(I+(x_m\mathbf{x}_E \st m\in A))=\beta_{i,i+j}(I),\quad\quad\text{for all $i$ and all $j>d+1$}.
\end{align}
  
  Now $$I_{\leq d}=(I+(x_m\bfx_E \st m\in A))_{\leq d} 
\ \ \ \ \mbox{ and }\ \ \ \  
I_{\leq d-1}=(I+(\bfx_E))_{\leq d-1}$$ so our assertions follows from \cref{betti numbers of ideals of deg<k}.
\end{proof}


  Recall that for a graded ideal $I$ of the polynomial
    ring $S$ the {\bf regularity} of $I$ 
    the maximum of all $j$ such that $\beta_{i,i+j}(I) \neq 0$.

  \begin{Corollary}[{\bf Adding generators to componentwise linear
        ideals}]\label{CWL+free:CWL} Let $I$ be a square-free monomial
    ideal in $K[x_1,\ldots,x_n]$, and suppose the degree of each
    minimal monomial generator of $I$ is $\leq d+1$. Let $E$ be a
    $(d-1)$-dimensional free face of $\Gamma:=\mathcal{N}(I)$
     or $\Delta_d(\Gamma)$, and $A\subseteq [n]-E$,  $A\neq \emptyset$,
    with $E\cup \{m\} \in \Gamma$ for each $m\in A$. If
    $I$ is componentwise linear, then $$I+(x_m\bfx_E: m\in A)$$ is
    componentwise linear of regularity $d+1.$
\end{Corollary}

\begin{proof}

  We show that $(I+(x_m\bfx_E: m\in A))_{[k]}$ has $k$-linear
  resolution for all $k$.

If $k<d+1$ we have $$(I+(x_m\bfx_E: m\in A))_{[k]}=I_{[k]},$$ and since
the latter has linear resolution, we are done.  Suppose $k\geq d+1$.  Then
\begin{align}\label{J}
(I+(x_m\bfx_E: m\in
  A))_{[k]}=\left(\mathfrak{m}^{k-d-1}(I_{[d+1]}+(x_m\bfx_E: m\in
  A))\right)^{sq},
\end{align}
where $\mathfrak{m}^{k-d-1}$ denotes $(k-d-1)$-st power of the graded
maximal ideal $\mathfrak{m}$ of $S$ and by $J^{sq}$ we mean the ideal
generated by square-free generators of $J$.

 By \cref{betti numbers of chordals}(b)  for all $i$ and all $j\neq d+1$
 $$\beta_{i,i+j}(I_{[d+1]}+(x_m\bfx_E: m\in
 A))=\beta_{i,i+j}(I_{[d+1]})=0.$$ Therefore $I_{[d+1]}+(x_m\bfx_E:
 m\in A)$ has a $(d+1)$-linear resolution. It follows
 from~\cite[Lemma~8.2.10]{HHBook}
 that $$J=\mathfrak{m}^{k-d-1}\left(I_{[d+1]}+(x_m\bfx_E: m\in
 A)\right)$$ has a $k$-linear resolution. Therefore, the square-free
 component $J^{sq}$ of $J$ has $k$-linear
 resolution~\cite[Proposition~8.2.17]{HHBook}, and so by \cref{J}
 $(I+(x_m\bfx_E: m\in A))_{[k]}$ has a $k$-linear resolution as
 desired.

By~\cite[Corollary~8.2.14]{HHBook} the regularity of the componentwise
linear ideal $I+(x_m\bfx_E: m\in A)$ is the highest degree of its
minimal generators, which in this case is equal to $d+1$.
\end{proof}

\begin{Theorem}[\bf Chordal complexes produce componentwise linear ideals]\label{main1}
  Let $I$ be a nonzero square-free monomial ideal,  $d$ a positive integer and let
  $\Gamma=\N(I)$. Then,  over all fields, we have 
	\begin{itemize}

	\item[{\rm(a)}] If $\Gamma$ is $d$-chordal, then
          $I_{[d+1]}=\N(\Delta_d(\Gamma))$ has a $(d+1)$-linear
          resolution (\cite[Theorem~3.3]{BYZ}).

        \item[\rm{(b)}] If $\Gamma$ is $d$-chordal and $W
          \subseteq[n]$, then $\N(\Gamma_W)_{[d+1]}$ has a
          $(d+1)$-linear resolution.

        \item[{\rm(c)}] If $\Gamma$ is $d$-collapsible, then
              $I_{[d+1]}=\N(\Delta_d(\Gamma))$ has $(d+1)$-linear
              resolution.

          \item[{\rm(d)}] If $\Gamma$ is $d$-representable, then
            $I_{[d+1]}=\N(\Delta_d(\Gamma))$ has $(d+1)$-linear
            resolution.

	 \item[{\rm(e)}]  If $\Gamma$ is chordal, then $I$ is
          componentwise linear.

        \item[\rm{(f)}] If $\Gamma$ is $d$-chordal for all $t-1\leq d
          \leq s-1$ where $t$ and $s$ are, respectively, the smallest
          and the largest degrees of the minimal monomial generators
          of $I$, then $I$ is componentwise linear.

          \item[{\rm(g)}] If $\Gamma$ is $d$-collapsible and $\deg u>
            d$ for all $u\in I$, then $I$ is componentwise linear.

          \item[{\rm(h)}] If $\Gamma$ is $d$-representable and $\deg
            u> d$ for all $u\in I$, then $I$ is componentwise linear.

          \item[\rm{(i)}] If $\Gamma$ is chordal and $W \subseteq[n]$,
            then $\N(\Gamma_W)$ is componentwise linear.
        \end{itemize}
	\end{Theorem}

\begin{proof}
	\begin{itemize}

\item[(a)] By \cref{l:prep-lemma}(a)
  $\N(I_{[d+1]})=\Delta_{d}(\Gamma)$. Since $\Gamma$ is $d$-chordal,
  $\Delta_d(\Gamma)$ admits a simplicial order $\mathbf{E}=E_1,\ldots,
  E_t$. It follows from \cref{betti numbers of chordals}(a) that for
  all $i$ and all $j>d+1$
$$\beta_{i,i+j}(I_{[d+1]})=\beta_{i,i+j}(I+(\bfx_{E_1},\ldots, \bfx_{E_t}))=\beta_{i,i+j}(\N(\langle
[n]\rangle^{[d-1]}))$$ The ideal $\N(\langle [n]\rangle^{[d-1]})$,
generated by all square-free monomials of degree $d+1$ in $S$, has
$(d+1)$-linear resolution over all fields (Herzog and
Hibi~\cite{HHPoly}).  Hence $\beta_{i,i+j}(I_{[d+1]})=0$ for all $i$
and all $j>d+1$.  Since $I_{[d+1]}$ is generated by monomials of
degree $d+1$, for each $i$, the $i$th syzygies are of degree $\geq
i+d+1$, and so $\beta_{i,i+j}(I_{[d+1]})=0$ for all $i$ and all
$j<d+1$.  Therefore $I_{[d+1]}$ has $(d+1)$-linear resolution over all
fields.

\item[(b)] Follows from Part~(a) and \cref{induced complexes of chordals}(e).

\item[(c)] Follows from Part~(a), \cref{d-collapsible=chordal}, and
\cref{d-collapsible skeleton}.

\item[(d)] Follows from Part~(a) and \cref{representable}.

\item[(e)] By assumption $\Gamma$ is $d$-chordal for all $d\geq
1$. Hence $I_{[d+1]}$ has $(d+1)$-linear resolution over all fields
using Part~(a). Since by \cite[Proposition~8.2.17]{HHBook} a square-free
monomial ideal $I$ is componentwise linear if and only if $I$ is
square-free componentwise linear, our assertion follows.

\item[(f)] Follows from Part~(e) and \cref{the bounds for checking chordality}.

\item[(g)] Follows from Part~(e) and  \cref{d-collapsible for suitable d}.

\item[(h)] Follows from Part~(g), and the fact that $d$-representable complexes are $d$-collapsible \cite{We}.

\item[(i)] Follows from Part~(e) and \cref{induced complexes of chordals}(f).

\end{itemize}
\end{proof}

Note that one can prove \cref{main1}(a) independently: Since
$\Gamma$ is $d$-chordal, $\Delta_d(\Gamma)$ is $d$-collapsible using
\cref{d-collapsible=chordal}. It is shown in \cite[Lemma~3]{We}
that any $d$-collapsible complex is $d$-Leray. Hence
$\wh_j(\Delta_d(\Gamma)_W;K)=0$ for all $j\geq d$. This
yields the desired conclusion.


The following example, which was suggested by Eric Babson in a
communication with Ali Akbar Yazdan Pour~\cite{BY}, shows that the
converses of none of the parts of \cref{main1} holds.

 \begin{Example}
   Let $\Gamma$ be a triangulation of a Dunce hat, see
   \cref{duncehat}, and let $\Sigma:=\Delta_2(\Gamma)$ be its
   $2$-closure. Then it is seen that $\Sigma$ is not $2$-collapsible,
   and hence it is not $2$-chordal or $2$-representable, while
   $\N(\Sigma)$ has $3$-linear resolution over all fields.

 \begin{figure}[ht!]
  \begin{center}
\begin{tikzpicture}[line cap=round,line join=round,>=triangle 45, scale=0.7]
\definecolor{pinky}{rgb}{0.6,0.2,0.}
\fill[color=black,fill=black,fill opacity=0.1] (8.,6.) -- (3.,-2.) -- (13.,-2.) -- cycle;
\draw (8.,6.)-- (3.,-2.);
\draw (3.,-2.)-- (13.,-2.);
\draw (13.,-2.)-- (8.,6.);
\draw (8.,2.)-- (7.,1.);
\draw (7.,1.)-- (7.46,0.);
\draw (7.46,0.)-- (8.68,0.);
\draw (8.68,0.)-- (9.,1.);
\draw (9.,1.)-- (8.,2.);
\draw (8.,2.)-- (8.,6.);
\draw (7.,1.)-- (8.,6.);
\draw (7.,1.)-- (6.09,2.9);
\draw (7.,1.)-- (4.8,0.97);
\draw (7.46,0.)-- (3.,-2.);
\draw (4.85,0.97)-- (7.46,0.);
\draw (7.46,0.)-- (6.48,-2.);
\draw (8.68,0.)-- (6.48,-2.);
\draw (8.,2.)-- (7.46,0.);
\draw (8.,2.)-- (8.68,0.);
\draw (8.68,0.)-- (10.,-2.);
\draw (8.68,0.)-- (13.,-2.);
\draw (9.,1.)-- (13.,-2.);
\draw (9.,1.)-- (11.1,0.98);
\draw (9.,1.)-- (9.87,3.0);
\draw (8.,2.)-- (9.87,3.0);
\begin{scriptsize}
\draw [fill](8.,6.) circle (1.5pt);
\draw (8.0,6.3) node {$1$};
\draw [fill] (3.,-2.) circle (1.5pt);
\draw (3.0,-2.35) node {$1$};
\draw [fill] (13.,-2.) circle (1.5pt);
\draw (13.0,-2.35) node {$1$};
\draw [fill] (6.08,2.9) circle (1.5pt);
\draw (5.75,3.0) node {$3$};
\draw [fill] (4.85,0.97) circle (1.5pt);
\draw(4.55,1.07) node {$2$};
\draw [fill] (11.14,0.97) circle (1.5pt);
\draw(11.5,1.07) node {$2$};
\draw [fill] (9.87,3.0) circle (1.5pt);
\draw(10.2,3.1) node {$3$};
\draw [fill] (6.48,-2.) circle (1.5pt);
\draw(6.48,-2.4) node {$3$};
\draw [fill] (10.,-2.) circle (1.5pt);
\draw (10.0,-2.4) node {$2$};
\draw [fill] (8.,2.) circle (1.5pt);
\draw (7.73,2.15) node {$6$};
\draw [fill] (7.,1.) circle (1.5pt);
\draw(6.6,1.23) node {$5$};
\draw [fill] (7.46,0.) circle (1.5pt);
\draw(7.5,-0.4) node {$4$};
\draw [fill] (8.68,0.) circle (1.5pt);
\draw (8.65,-0.45) node {$8$};
\draw [fill] (9.,1.) circle (1.5pt);
\draw (9.4,1.23) node {$7$};
\end{scriptsize}
\end{tikzpicture}
\caption{A triangulation of the Dunce hat} \label{duncehat}
\end{center}	
\end{figure}
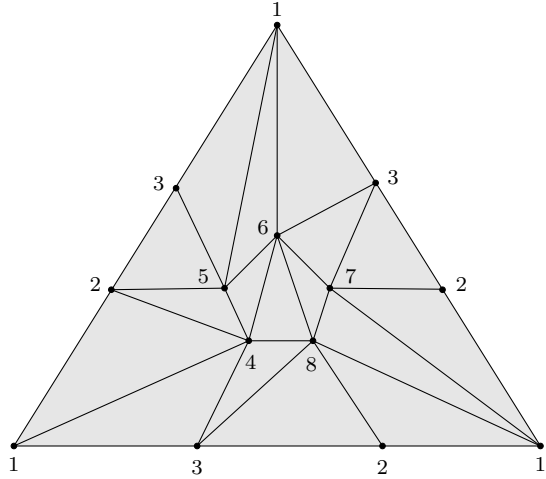	
 \end{Example}

In the next section we show that the Betti numbers of all
componentwise linear ideals appear as Betti numbers of Stanley-Reisner
ideals of chordal complexes.

\subsubsection*{Cohen-Macaulay properties}

Let $\Gamma$ be a simplicial complex on the vertex set $[n]$,
$I=\N(\Gamma)$ be an ideal of $S=K[x_1,\ldots,x_n]$ where $K$ is a
field, and let $K[\Gamma]=S/I$ be the Stanley-Reisner ring of
$\Gamma$.

A pure complex $\Gamma$ is called {\bf Cohen-Macaulay} over $K$ if $K[\Gamma]$
is a Cohen-Macaulay ring, or, equivalently, by Eagon and
Reiner~\cite[Theorem~3]{ER}, if $I^\vee=\N(\Gamma^\vee)$ has a linear
resolution.

Stanley~\cite{S} generalized the Cohen-Macaulay property to all
simplicial complexes, calling this new class of complexes {\bf
  sequentially Cohen-Macaulay}. In Duval's~\cite{D} characterization,
the complex $\Gamma$ is sequentially Cohen-Macaulay over $K$ if and
only if $\Gamma^{[d]}$ is Cohen-Macaulay (over $K$) for all $d\leq
\dim \Gamma$ .  Herzog and Hibi~\cite[Theorem~2.1(a)]{HHCWL} then
extended the criterion of Eagon and Reiner showing that a square-free
monomial ideal $I$ is componentwise linear if and only if $\N(I)^\vee$ is
sequentially Cohen-Macaulay.

Combining these facts with \cref{CWL+free:CWL} and \cref{main1}, we
make the following observation.

\begin{Corollary}[\bf Chordal complexes have sequentially Cohen-Macaulay duals]
  Let $\Gamma$ be a simplicial complex on $[n]$.   If $\Gamma$ is
  either $d$-chordal or $d$-collapsible or $d$-representable, then
  $(\Gamma^\vee)^{[n-d-2]}$ is Cohen-Macaulay. In particular, if
  $\Gamma$ is chordal, then $\Gamma^\vee$ is sequentially
  Cohen-Macaulay.
\end{Corollary}

\begin{proof}
Setting $I=\N(\Gamma)$, it follows from \cite[page 131]{F} that  
    $$I_{[d+1]}=\N(((\Gamma^\vee)^{[n-d-2]})^\vee).$$ 
    Hence $I_{[d+1]}$ has linear resolution if and only if $(\Gamma^\vee)^{[n-d-2]}$ is Cohen-Macaulay. Our statements now follow from \cref{main1}.

     \end{proof}


\section{More chordal complexes and Betti tables of componentwise linear ideals}

   In this section we focus on well-known classes of componentwise
   linear ideals with, and of simplicial complexes which arise from
   them. It is still not known whether Alexander duals of shellable
   complexes (Bj\"orner and Wachs~\cite{BWI}), which provide a large
   class of componentwise linear ideals containing most other such
   ideals, are chordal  (see Herzog and Hibi~\cite{HHCWL},
   and also Eagon and Reiner~\cite{ER}).

We also show in this section that the Betti table of every
componentwise linear ideal is equal to that of the Stanley-Reisner
ideal of a chordal complex.


\subsection{Alexander duals of vertex decomposable
    complexes}

  One large class of ideals with linear resolution is the class
  of Stanley-Reisner ideals of the Alexander duals of vertex
  decomposable complexes  (Bj\"orner and
    Wachs~\cite{BWII}, Provan and Billera~\cite{PB}).
    Nikseresht~\cite{Nik} showed that if a pure $d$-dimensional
    simplicial complex on $n$ vertices is vertex decomposable, then
    its Alexander dual is $(n-d-2)$-chordal. Here we use this result
    to show that the Alexander dual of any vertex decomposable
    simplicial complex is chordal.

    The main idea is  that, similar to the property of sequential
    Cohen-Macaulayness, vertex decomposability of a simplicial complex
    reduces to that of its pure skeletons, \cite[Lemma~3.10]{Wo}.

  \begin{Definition}[{\bf vertex decomposable simplicial complex}] A
    simplicial complex $\Gamma$ on the vertex set $[n]$ is called {\bf
      vertex decomposable} if it is a simplex, including
      $\emptyset$ and $\{\emptyset\}$, or it contains a vertex $v$
      such that
 
 \begin{itemize}
\item[(i)] $v$ is a {\bf shedding vertex} of $\Gamma$, i.e no face of  $\link_\Gamma (v)$  is a facet of $\Gamma \sm \{v\}$, and 
\item[(ii)] both $\Gamma \sm \{v\}$ and $\link_\Gamma(v)$ are vertex decomposable.
\end{itemize}
\end{Definition}

Nikseresht~\cite[Lemma~3.1]{Nik} shows that for a pure $d$-dimensional
complex $\Delta$, a vertex $v$ is a shedding vertex if and only if
$\Delta \sm \{v\}$ is also pure of dimension $d$. This fact will be used
in the arguments below.

 \begin{Theorem}[{\bf Alexander duals of vertex decomposable complexes are chordal}]\label{VD->Chordal}
 Let $\Gamma$ be a vertex decomposable complex on $[n]$. Then its Alexander dual $\Gamma^\vee$ is chordal.
 \end{Theorem}
 \begin{proof} Let $d \geq 1$. We need to show that
     $\Delta_{d}(\Gamma^\vee)$ is $d$-chordal. From~\cite[page 131]{F}, we 
     have that $$\N((\Gamma^{[n-d-2]})^{\vee})=\N(\Gamma^\vee)_{[d+1]}$$ which by
\cref{l:prep-lemma} implies that 
 $$(\Gamma^{[n-d-2]})^{\vee}=\Delta_{d}(\Gamma^\vee).$$

Woodroofe proves in \cite[Lemma~3.10]{Wo} that all the skeletons of a vertex decomposable simplicial complex are vertex decomposable. Since $\Gamma$ is vertex decomposable, it follows that 
$\Gamma^{[n-d-2]}$ is vertex decomposable too. On the other hand
Nikseresht~\cite[Theorem~3.10]{Nik} proved that  the dual of any
pure $t$-dimensional vertex decomposable complex is
$(n-t-2)$-chordal. Therefore
$\left(\Gamma^{[n-d-2]}\right)^\vee=\Delta_{d}(\Gamma^\vee)$ is
$d$-chordal, as desired.
  \end{proof}


 \subsection{Square-free (strongly) stable ideals}

Square-free stable ideals, defined by Aramova, Herzog and
Hibi~\cite{AHH} form a large class of componentwise linear ideals.
This class contains the class of square-free strongly stable ideals
and lexsegment ideals.  

 For a monomial $u\in S=K[x_1,\ldots,x_n]$ we define $m(u)=\max\{i \st
 x_i \mid u \}$. A square-free monomial ideal $I$ is called {\bf
   square-free stable} if for all square-free monomials $u\in I$
 $$x_i \left ( \frac{u}{x_{m(u)}} \right )\in I \mbox{ for all }
 i<m(u) \mbox{ such that } x_i \nmid u,$$ and $I$ is called {\bf
   square-free strongly stable} if for all square-free monomials $u\in
 I$ and $x_j \mid u$
 $$x_i \left ( \frac{u}{x_{j}} \right )\in I  
 \mbox{ for all } i<j \mbox{ such that } x_i \nmid u.$$ 
It turns out that the defining property for
 square-free (strongly) stable ideals $I$ needs only be checked for the monomials in
 the minimal monomial generating set
 $\mathcal{G}(I)$~\cite[Problem~6.9]{HHBook}.

 \begin{Theorem}[{\bf Stanley-Reisner complexes of square-free stable
 ideals are chordal}]\label{SS} Let $I$ be a square-free stable ideal
   in $S=K[x_1,\ldots,x_n]$, $K$ a field. Then $\N(I)$ is chordal.
 \end{Theorem}

 \begin{proof}
 First note that for each $d \geq 1$ the ideal $I_{[d+1]}$ is square-free
 stable, for if $u\in \mathcal{G}(I_{[d+1]})\subseteq I$ and $i<m(u)$
 with $x_i\not|u$, the monomial $x_i(u/x_{m(u)})\in I$ and $\deg
 (x_i(u/x_{m(u)}))=d+1$ which implies that $x_i(u/x_{m(u)})\in
 I_{[d+1]}$. By Nikseresht and Zaare-Nahandi's
 work~\cite[Theorem~2.5]{NZ} the complex
 $\Delta_d(\N(I))=\N(I_{[d+1]})$ is $d$-chordal. Therefore $\N(I)$ is
 chordal.
  \end{proof}

Recall that a simplicial complex $\Gamma$ is called {\bf shifted} if
for any face $F\in \Gamma$, any $i\in F$ and $j\in [n]$ with $j>i$ one
has $(F-\{i\})\cup\{j\}\in \Gamma$.

\Cref{SS} in particular implies that square-free strongly stable
ideals have chordal Stanley-Reisner complexes. This statement can also
be deduced from the fact that and ideal $I$ is square-free strongly
stable if and only if $\left(\N(I)\right)^\vee$ is shifted, and
therefore vertex decomposable by~\cite[Theorem~11.3]{BWII}. Hence
$\N(I)$ is chordal by \cref{VD->Chordal}. 

\medskip
   We now show that the study of the Betti tables of componentwise
   linear ideals reduces to the study of the Betti tables of
   Stanley-Reisner ideals of chordal complexes, generalizing a similar
   result of Bigdeli and coauthors in the case of equigenerated
   ideals~\cite[Theorem~3.3]{BHYZ}.
  
  For the proof we use  the {\bf square-free
      operator}~\cite{HHBook} which takes a monomial
    $u=x_{i_1}x_{i_2}\cdots x_{i_t} \in S$ with $i_1\leq \cdots\leq
    i_t$, to the square-free monomial
    $u^{\sigma}=x_{i_1}x_{i_2+1}\cdots x_{i_t+(t-1)}$. If $I$ is a
    monomial ideal with $\mathcal{G}(I)=\{u_1,\ldots,u_m\}$, then
    $I^{\sigma}$ is the square-free monomial ideal $$I^{\sigma}=
    (u_1^{\sigma},\ldots,u_m^{\sigma}).$$
    
 \begin{Theorem}[{\bf Chordal complexes give Betti tables of 
all componentwise linear ideals}]\label{BT-CWL} Let $K$ be a field and
   $I\subset S=K[x_1,\ldots,x_n]$ be a graded ideal which is
   componentwise linear. Then there exists a chordal complex $\Gamma$
   such that the Betti table of $I$ coincides with that of
   $\N(\Gamma)$.
 \end{Theorem}

 \begin{proof}
 It follows from Herzog and coauthors~\cite[Proposition~2.1]{HShV}
 that the Betti table of a componentwise linear ideal $I$ coincides
 with the Betti table of a strongly stable ideal $J$ (not necessarily
 square-free).  By \cite[Lemma~11.2.5]{HHBook} $J^{\sigma}$ is
 square-free strongly stable. Moreover, \cite[Lemma~11.2.6]{HHBook}
 implies that $J^{\sigma}$ has the same Betti table as of $J$. Hence
 the Betti tables of $I$ and $J^\sigma$ coincide. Since square-free
 strongly stable ideals are square-free stable,
 \cref{SS} implies that $\N(J^{\sigma})$ is chordal, as desired.
  \end{proof}
 

 \subsection{Square-free Gotzmann ideals}

A homogeneous ideal $I$ in a polynomial ring $S=K[x_1,\ldots,x_n]$
over a field $K$ is a {\bf Gotzmann ideal} if its ``growth'' in
degrees is similar to a lex ideal. More precisely, let $S_1$ be the
first graded piece of $S$ (generated by $x_1,\ldots,x_n$ as a
$K$-vector space), and similarly, let $I_u$ be the $u$-th graded piece
of $I$ (generated by all degree $u$ monomials in $I$), and $L$ be a
lex ideal with the same Hilbert function as $I$.  Then $I$ is Gotzmann
if and only if $\dim_K(S_1I_u)=\dim_K(S_1L_u)$ for all $u \geq 0$.

Herzog and Hibi~\cite{HHCWL} proved that Gotzmann monomial ideals are
componentwise linear. Below we use a characterization of Gotzmann
square-free monomial ideals due to Hoefel and Mermin~\cite{HM} to show
that the Stanley-Reisner complex of these ideals is chordal.

 \begin{Theorem}[Hoefel~\cite{Ho}, Theorem 5.9;
   Hoefel-Mermin~\cite{HM}, Theorem 3.9]\label{t:HMG} Let $K$ be a field,
   $S=K[x_1,\ldots,x_n]$ and ideal $I$ be a square-free monomial ideal in
   $S$. Then $I$ is a Gotzmann ideal if and only if  $I$
     is generated by one variable or
   $$I=m_1(z_{1,1},\ldots,z_{1,r_1}) +
   m_1m_2(z_{2,1},\ldots,z_{2,r_2}) + \cdots + m_1m_2\cdots
   m_s(z_{s,1},\ldots,z_{s,r_s})$$ for some square-free monomials
   $m_1,\ldots,m_s$ and variables $z_{i,j}$ all having pairwise
   disjoint support and satisfying
\begin{itemize}
\item $\deg (m_i) \geq 1$ for $1< i \leq s$,
\item $r_i \geq 1$ for  $1\leq  i < s$,
\item $r_s \neq 1$ and
\item $\deg(m_s) \geq 2$ when $r_s = 0$.
\end{itemize}
\end{Theorem}

\begin{Theorem}[{\bf Gotzmann ideals are chordal}]\label{Gotz}
  Let $I$ be a Gotzmann square-free monomial ideal in
  $S=K[x_1,\ldots,x_n]$, $K$ a field. Then $\N(I)$ is chordal.
  \end{Theorem}

  \begin{proof} With notation as in \cref{t:HMG},
    let $$I=m_1(z_{1,1},\ldots,z_{1,r_1}) +
    m_1m_2(z_{2,1},\ldots,z_{2,r_2}) + \cdots + m_1m_2\cdots
    m_s(z_{s,1},\ldots,z_{s,r_s}),$$ where $m_i=y_{i,1} \cdots
    y_{i,t_i}$ for $1 \leq i \leq s$, and all the $z_{i,j}$ and
    $y_{i,j}$ are distinct variables in $\{x_1,\ldots,x_n\}$. 

    Now we re-order the variables, so that setting
    $\alpha_i=\displaystyle \sum^{{i-1}}_{j=1} t_j+r_j$ for $i>1$ and
    $\alpha_1=0$,  for $1 \leq i\leq s$ we have 
     $$x_{\alpha_i +1}=y_{i,1}, \ldots, x_{\alpha_i + t_i}=y_{i,t_i},   \ \ \ \  x_{\alpha_i + t_i+1}=z_{i,1}, \ldots, x_{\alpha_i+ t_i+r_i}=z_{1,r_i}.$$

    So the relabeled form of $I$ is
    $$m'_1(x_{t_1+1},\ldots,x_{t_1+r_1}) + \cdots + m'_1m'_2\cdots
    m'_s(x_{\alpha_s+t_s+1},\ldots,x_{\alpha_s+t_s+r_s}),$$ where
    $m'_i=x_{\alpha_i+1} \cdots x_{\alpha_i+t_i}$ for $1 \leq i \leq
    s$.

This latter ideal is clearly square-free strongly stable. To see this,
take any monomial generator of the form $M=m'_1m'_2\cdots m'_v
x_{\alpha_v+t_v+u}$. Suppose $x_i|M$, $j<i$ and $x_j \nmid M$. Then
$j=\alpha_w+t_w+l$, where $w\leq v$ and $\left \{
         \begin{array}{ll} 
          1\leq l<u & \text{if }w=v,\\
          1\leq l\leq r_w & \text{if }w<v.
         \end{array}\right.$

    Since $m'_1m'_2\cdots m'_w x_{\alpha_w+t_w+l}$ is a generator, the
    monomial $x_j(M/x_i)=m'_1m'_2\cdots m'_v x_{\alpha_w+t_w+l}$
    belongs to $I$, and  we are done.

    Now $\N(I)$ is isomorphic to the Stanley-Reisner complex of a
    square-free strongly stable ideal, and is therefore chordal by
    \cref{SS}.
  \end{proof}

\Cref{Gotz} can also be proved directly, because of the nice inductive structure that square-free Gotzmann ideals have.

\section{Further questions and remarks  }

\begin{Remark} 
  It is well-known~\cite{Di} that any chordal graph has at least two
  simplicial vertices.  Equivalently, the flag complex of a chordal
  graph (which is a $1$-closure) has at least two simplicial faces
  which are not facets.  One may ask if the same holds for the
  $d$-closure of an arbitrary $d$-chordal simplicial complex.
   Theorem~2.3 of~\cite{ABH} implies that for any $d>1$
  there is a $d$-dimensional simplicial complex $\Gamma$ which is
  $d$-collapsible and has only one free face $E$ of dimension $d-1$
  which is not a facet.  It turns out that $E$, being contained in a
  single $d$-dimensional facet of $\Gamma$, is a simplicial face of
  $\Delta_d(\Gamma)$ which is not a facet.  By \cref{d-collapsible
    skeleton}, we know that $\Delta_d(\Gamma)$ is $d$-collapsible.
  Now \cref{d-collapsible=chordal} implies that $\Delta_d(\Gamma)$ is
  $d$-chordal with $E$ as its non-facet simplicial face.

  \Cref{a complex with one free face} is an example of the complexes constructed in
   Theorem~2.3 of~\cite{ABH}. It is a $2$-dimensional
  $2$-collapsible complex $\Gamma$ with $\{1,2\}$ as its unique free
  face. Then $\Delta_2(\Gamma)=\Gamma\cup\{\{3,5\}, \{5,7\}\}$ is
  $2$-chordal with $\{1,2\}$ as a simplicial face, by above
  argument. It is easy to check that indeed, $\{1,2\}$ is the unique
  non-facet simplicial face of the complex $\Delta_2(\Gamma)$. So the
  answer to the above question is negative in general.

			\begin{figure}[ht!]
			\begin{center}
			\begin{tikzpicture}[line cap=round,line join=round,>=triangle 45,x=0.8cm,y=0.8cm]
\clip(1.,1.2) rectangle (7.,6.8);
\fill[line width=1pt,fill=black,fill opacity=0.10000000149011612] (2.,2.) -- (6.,2.) -- (6.,6.) -- (2.,6.) -- cycle;
\fill[line width=1pt,fill=black,fill opacity=0.10000000149011612] (3.,4.) -- (4.,5.) -- (5.,4.) -- (4.,3.) -- cycle;
\draw [line width=1.pt] (2.,2.)-- (6.,2.);
\draw [line width=1.pt] (6.,2.)-- (6.,6.);
\draw [line width=1.pt] (6.,6.)-- (2.,6.);
\draw [line width=1.pt] (2.,6.)-- (2.,2.);
\draw [line width=1.pt] (2.,4.)-- (6.,4.);
\draw [line width=1.pt] (3.,4.)-- (4.,5.);
\draw [line width=1.pt] (4.,5.)-- (5.,4.);
\draw [line width=1.pt] (5.,4.)-- (4.,3.);
\draw [line width=1.pt] (4.,3.)-- (3.,4.);
\draw [line width=1.pt] (4.,5.)-- (4.,6.);
\draw [line width=1.pt] (4.,5.)-- (2.,6.);
\draw [line width=1.pt] (3.,4.)-- (2.,6.);
\draw [line width=1.pt] (4.,5.)-- (6.,6.);
\draw [line width=1.pt] (5.,4.)-- (6.,6.);
\draw [line width=1.pt] (5.,4.)-- (6.,2.);
\draw [line width=1.pt] (6.,2.)-- (4.,3.);
\draw [line width=1.pt] (4.,3.)-- (2.,2.);
\draw [line width=1.pt] (2.,2.)-- (3.,4.);
\draw (1.55,6.46) node[anchor=north west] {\begin{scriptsize}$2$\end{scriptsize}};
\draw (3.74,6.55) node[anchor=north west] {\begin{scriptsize}$3$\end{scriptsize}};
\draw (5.94,6.46) node[anchor=north west] {\begin{scriptsize}$1$\end{scriptsize}};
\draw (3.86,5.6) node[anchor=north west] {\begin{scriptsize}$7$\end{scriptsize}};
\draw (1.55,4.3) node[anchor=north west] {\begin{scriptsize}$3$\end{scriptsize}};;
\draw (2.75,4.65) node[anchor=north west] {\begin{scriptsize}$4$\end{scriptsize}};
\draw (4.7,4.65) node[anchor=north west] {\begin{scriptsize}$6$\end{scriptsize}};
\draw (5.94,4.3) node[anchor=north west] {\begin{scriptsize}$3$\end{scriptsize}};
\draw (3.75,3.) node[anchor=north west] {\begin{scriptsize}$5$\end{scriptsize}};
\draw (1.55,2.2) node[anchor=north west] {\begin{scriptsize}$1$\end{scriptsize}};
\draw (5.94,2.2) node[anchor=north west] {\begin{scriptsize}$2$\end{scriptsize}};
\begin{scriptsize}
\draw [fill=black] (2.,2.) circle (1.5pt);
\draw [fill=black] (6.,2.) circle (1.5pt);
\draw [fill=black] (6.,6.) circle (1.5pt);
\draw [fill=black] (2.,6.) circle (1.5pt);
\draw [fill=black] (2.,4.) circle (1.5pt);
\draw [fill=black] (6.,4.) circle (1.5pt);
\draw [fill=black] (3.,4.) circle (1.5pt);
\draw [fill=black] (4.,5.) circle (1.5pt);
\draw [fill=black] (5.,4.) circle (1.5pt);
\draw [fill=black] (4.,3.) circle (1.5pt);
\draw [fill=black] (4.,6.) circle (1.5pt);
\end{scriptsize}
\end{tikzpicture}
\caption{A $2$-chordal complex with $\{1,2\}$ as the unique simplicial (non-facet) face of its $2$-closure}
			\label{a complex with one free face}
			\end{center}
			\end{figure}
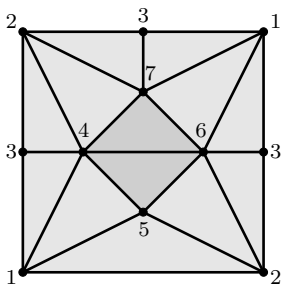
\end{Remark}


\begin{Remark} The class of $d$-chordal complexes includes
      nonshellable ones.  Setting $\Gamma$ to be the triangulation of
      the dunce hat in \cref{duncehat}, it is well known that $\Gamma$
      is a Cohen-Macaulay non-shellable complex while
      \cite[Example~3.14]{BYZ} implies that $\Gamma^\vee$, which is a
      $4$-closure, is chordal.
    \end{Remark}

The following question is then a natural one.

\begin{Question} A large combinatorial class of simplicial complexes
  whose Stanley-Reisner ideals are componentwise linear are Alexander
  duals of shellable complexes.  Are duals of shellable complexes
  chordal?  Since shellability reduces to the pure skeletons \cite[Theorem~2.9]{BWI}, it is enough to ask the question in pure case. 
    Equivalently one can ask: is the Stanley-Reisner complex of an 
    ideal equigenerated  in degree $d+1$ with linear quotients (\cite{HT}) $d$-chordal? (See
    also~\cite[Proposition~8.2.5]{HHBook}.)
\end{Question}

\begin{Question} Not all free faces of a $d$-collapsible
    complex $\Gamma$ can be the starting face of a free sequence which
    reduces $\Gamma$ to $\emptyset$. Tancer~\cite{T} constructs
    $d$-collapsible complexes $\Gamma$ with a free face $E$ (which he
    calls a ``bad'' face) such that $\Gamma\coll{E}$ is not
    $d$-collapsible.  What about the case of $d$-chordal complexes:
    given a $d$-chordal $d$-closure $\Gamma$ and a simplicial face
    $E$, is $\Gamma\sr{E}$ always $d$-chordal?
\end{Question}



\begin{thebibliography}{}
 
\bibitem{ABH} K.~Adiprasito, B.~Benedetti, and F.~ H.~Lutz. \textit{Extremal examples of collapsible complexes and random discrete Morse theory}, Discrete \& Computational Geometry, \textbf{57.4},  824 -- 853 (2017).

\bibitem{ANS} K.~A.~Adiprasito E.~Nevo, and J.~A.~Samper, \textit{Higher chordality: From graphs to complexes},  Proc. Amer. Math. Soc., \textbf{144.8},  3317 -- 3329 (2016).

\bibitem{AHH} A.~Aramova, J.~Herzog, and T.~Hibi, \emph{Squarefree  lexsegment ideals}, Math. Z., {\bf 228}, no. 2, 353  --  378 (1998).

\bibitem{BY} E.~Babson and A.~A.~Yazdan Pour, private communication (2017).

\bibitem{BHYZ}M.~Bigdeli, J.~Herzog, A.~A.~Yazdan Pour, and R.~Zaare-Nahandi, \textit{Simplicial orders and Chordality}, J.~Algebraic Combin., \textbf{45.4}, 1021 -- 1039 (2017).

\bibitem{BiY} M.~Bigdeli and A.~A.~Yazdan Pour, \textit{Multigraded minimal resolution of simplicial subclutters}, preprint (2018).

\bibitem{BYZ} M.~Bigdeli, A.~A.,~Yazdan Pour, and R.~Zaare-Nahandi, \textit{Stability of Betti numbers under reduction processes: Towards chordality of clutters}, J.~Combin.~Theory Ser.~A, \textbf{145}, 129 -- 149 (2017).

\bibitem{BWI} A.~Bj\"orner and M.~Wachs, \emph{Shellable nonpure  complexes and posets I}, Trans. Amer. Math. Soc. {\bf 348}, no. 4, 1299  --   1327 (1996).

\bibitem{BWII} A.~Bj\"orner and M.~Wachs, \emph{Shellable nonpure complexes and posets II}, Trans. Amer. Math. Soc., {\bf 349}, no. 10, 3945  --  3975 (1997).

\bibitem{CF1}E.~Connon and S.~Faridi, \textit{Chorded complexes and a necessary condition for a monomial ideal to have a linear resolution}, J.~Combin.~Theory~Ser.~A,  \textbf {120},  no.~7, 1714 -- 1731 (2013).

\bibitem{CF2} E.~Connon and S.~Faridi, \textit{A criterion for a monomial ideal to have a linear resolution in characteristic 2}, Electron.~J.~Combin., \textbf{22}, no.~1, Paper 1.63, 15 pages  (2015).

\bibitem{CLS} R.~Cordovil, M.~Lemos, and C.~Sales, \emph{Dirac'€™s  Theorem on Simplicial Matroids}, Ann.~Comb., \textbf{13}, 53 -- 63 (2009)

\bibitem{Di} G.~A.~Dirac, \textit{On rigid circuit graphs}, Abh.~Math.~Sem.~Univ.~Hamburg, \textbf{25},  71 -- 76 (1961).

\bibitem{D} A.~Duval, \emph{Algebraic shifting and sequentially Cohen-Macaulay simplicial complexes}, Electron.~J.~Combin., \textbf{3},  no.  1, Research Paper 21 (1996).

\bibitem{ER} J. A.~Eagon and V.~Reiner, \emph{Resolution of  Stanley-Reisner rings and Alexander duality}, J. Pure Appl. Algebra,  {\bf 130}, no. 3, 265 -- 275 (1998).
  
\bibitem{Em}E.~Emtander, \textit{A class of hypergraphs that generalizes chordal graphs}, Math.~Scand., \textbf{106}, no. 1,  50 -- 66 (2010).

\bibitem{F} S.~Faridi, \emph{Simplicial trees are sequentially    Cohen-Macaulay}, J. Pure Appl. Algebra, {\bf 190}, no.~1-3, 121  --  136  (2004).

\bibitem{Fr} R.~Fr\"{o}berg, \textit{On Stanley-Reisner rings}, in: Topics in algebra, Banach Center Publications, \textbf{26} Part 2, 57  --  70 (1990).
 
\bibitem{He} E.~Helly, \textit{\"Uber Mengen konvexer K\"orper mit gemeinschaftlichen Punkten}, J.-Ber.~Deutsch.~Math.-Verein., \textbf{32}, 175 -- 176 (1923).

\bibitem{HHCWL} J.~Herzog and T.~Hibi, \textit{Componentwise Linear Ideals}, Nagoya Math. J., \textbf{153}, 141 -- 153 (1999).

\bibitem{HHPoly} J.~Herzog and T.~Hibi, \textit{Cohen-Macaulay    polymatroidal ideals}, European J. Combin.,  {\bf 27}(4), 513--517 (2006).

\bibitem{HHBook} J.~Herzog and T.~Hibi, \textit{Monomial Ideals}, in: GTM \textbf{260}, Springer, London (2011).

\bibitem{HShV} J.~Herzog, , L.~Sharifan, and M.~Varbaro. \textit{The    possible extremal Betti numbers of a homogeneous ideal}, Proc.~Amer.~Math.~Soc., {\bf142.6}, 1875 -- 1891 (2014).

\bibitem{HT} J.~Herzog and Y.~Takayama, \textit{Resolutions by mapping cones}, The Roos Festschrift, vol. 2, Homology, Homotopy Appl., {\bf 4} (2,part 2) 277 -- 294 (2002).

\bibitem{Hochster} M.~Hochster, \textit{Cohen-Macaulay rings, combinatorics, and simplicial complexes}, in: Ring Theory, II, Proc.~Second Conf., Univ. Oklahoma, Norman, Okla., 1975, in: Lecture Notes in Pure and Appl. Math., vol. 26, Dekker, New York, 171 -- 223 (1977).

\bibitem{Ho} A. Hoefel, \emph{Hilbert Functions in Monomial Algebras},  Ph.D. thesis, Dalhousie University, available at  \emph{dalspace.library.dal.ca} (2011).


\bibitem{HM}  A. Hoefel, J. Mermin, \emph{ Gotzmann squarefree ideals}, Illinois J. Math., {\bf 56}, no. 2, 397-- 414 (2012).

\bibitem{K} D.~Kosolv, \emph{Combinatorial Algebraic Topology},  Algorithms and Computation in Mathematics, {\bf 21}, Springer, Berlin,  2008.

\bibitem{LB} C.~G.~Lekerkerker and J.~C.~Boland,  \textit{Representation of a finite graph by a set of intervals on    the real line}, Fund.~Math., \textbf{51}, 45 -- 64 (1962).

\bibitem{MNYZ} M.~Morales, A.~Nasrollah Nejad, A.~A.~Yazdan Pour, and  R.~Zaare-Nahandi, \textit{Monomial ideals with $3$-linear    resolutions}, Annales de la Facult\'e des Sciences de Toulouse,  S\'er.~6, \textbf{23}: (4), 877 -- 891 (2014).

\bibitem{Nik} A.~Nikseresht, \textit{Chordality of clutters with    vertex decomposable dual and ascent of Clutters}, preprint (2017) \href{http://arxiv.org/abs/1708.07372v1}{\texttt{arXiv:1708.07372v1}}

\bibitem{NZ}  A.~Nikseresht and R.~Zaare-Nahandi, \textit{On generalization of cycles and chordality to clutters from an algebraic viewpoint}, Algebra~Colloq., \textbf{24}, 611 (2017).

\bibitem{PB} J.S.~Provan and L.J.~Billera, \emph{Decompositions of    simplicial complexes related to diameters of convex polyhedra},  Math. of Operations Research 5 (1980), 576–594. MR 82c:52010.

\bibitem{S} R.~Stanley, \emph{Combinatorics and commutative
algebra} , Second edition.  Progress in Mathematics, 41.
Birkhuser Boston, Inc., Boston, MA,   x+164 pp.  ISBN: 0-
8176-3836-9 (1996).


\bibitem{T} M.~Tancer, \emph{$d$-collapsibility is NP-complete for    $d\geq 4$}, Chic. J. Theoret. Comput. Sci., Article 3, (2010).

\bibitem{VTV} A.~Van~Tuyl and R.~H.~Villarreal, \textit{Shellable    graphs and sequentially Cohen-Macaulay bipartite    graphs}. J.~Combin.~Theory Ser. A, \textbf{115(5)}, 799 -- 814  (2008).

\bibitem{We} G.~Wegner, \textit{$d$-Collapsing and nerves of families    of convex sets}, Archiv~der~Mathematik, \textbf{26(1)}, 317 -- 321  (1975).

\bibitem{Wo} R.~Woodroofe, \textit{Chordal and sequentiallyCohen-Macaulay clutters}, Electron.~J.~Combin. {\textbf {18}},  no. 1, Paper 208, 20 pages (2011).

\end{thebibliography}
\end{document}